\documentclass[12pt]{article}
\usepackage{amsmath,amssymb,amsfonts,amsthm}
\newenvironment{myabstract}{\par\noindent
{\bf Abstract . } \small }
{\par\vskip8pt minus3pt\rm}
\newcounter{item}[section]
\newcounter{kirshr}
\newcounter{kirsha}
\newcounter{kirshb}
\newenvironment{enumroman}{\setcounter{kirshr}{1}
\begin{list}{(\roman{kirshr})}{\usecounter{kirshr}} }{\end{list}}
\newenvironment{enumarab}{\setcounter{kirshb}{1}
\begin{list}{(\arabic{kirshb})}{\usecounter{kirshb}} }{\end{list}}
\newtheorem{theorem}{Theorem}[section]

\newtheorem{lemma}[theorem]{Lemma}
\newtheorem{corollary}[theorem]{Corollary}
\newenvironment{demo}[1]{\noindent{\bf #1.}\upshape\mdseries}
{\nopagebreak{\hfill\rule{2mm}{2mm}\nopagebreak}\par\normalfont}
\theoremstyle{definition}

\newtheorem{example}[theorem]{Example}
\newtheorem{definition}[theorem]{Definition}
\def\R{\mathbb{R}}

\def\C{{\mathfrak{C}}}

\def\At{{\bf At}}
\def\Nr{{\mathfrak{Nr}}}

\def\Sg{{\mathfrak{Sg}}}

\def\A{{\mathfrak{A}}}
\def\B{{\mathfrak{B}}}
\def\C{{\mathfrak{C}}}
\def\D{{\mathfrak{D}}}
\def\M{{\mathfrak{M}}}
\def\N{{\mathfrak{N}}}

\def\CA{{\bf CA}}

\def\SC{{\bf SC}}
\def\QEA{{\bf QEA}}

\def\Lf{{\bf Lf}}

\def\PA{{\bf PA}}
\def\PEA{{\bf PEA}}

\def\K{{\bf K}}
\def\K{{\bf K}}

\def\RCA{{\bf RCA}}

\def\Rd{{\ Rd}}
\def\(R)RA{{\bf (R)RA}}
\def\RA{{\bf RA}}
\def\RRA{{\bf RRA}}

\def\R{\mathbb{R}}

\def\c #1{{\cal #1}}
 \def\CA{{\sf CA}}
\def\B{{\sf B}}
\def\G{{\sf G}}
\def\w{{\sf w}}
\def\y{{\sf y}}
\def\g{{\sf g}}

\def\r{{\sf r}}
\def\K{{\sf K}}

 \def\Cm{{\mathfrak{Cm}}}
\def\Nr{{\mathfrak{Nr}}}
\def\SNr{{\bf S}{\mathfrak{Nr}}}

\def\cyl#1{{\sf c}_{#1}}

\def\Ra{{\mathfrak{Ra}}}
\def\Ca{{\mathfrak{Ca}}}
\def\set#1{\{#1\} }
\def\Ra{{\mathfrak{Ra}}}
\def\Nr{{\mathfrak{Nr}}}
\def\Tm{{\mathfrak{Tm}}}
\def\A{{\mathfrak{A}}}
\def\B{{\mathfrak{B}}}
\def\C{{\mathfrak{C}}}
\def\D{{\mathfrak{D}}}

\def\A{{\mathfrak{A}}}
\def\B{{\mathfrak{B}}}
\def\C{{\mathfrak{C}}}
\def\D{{\mathfrak{D}}}

\def\GG{{\mathfrak{GG}}}

\def\L{{\mathfrak{L}}}
\def\Rd{{\mathfrak{Rd}}}

\def\At{{\mathfrak{At}}}
\def\L{{\mathfrak{L}}}
\def\Bl{{\mathfrak{Bl}}}
\def\CA{{\bf CA}}
\def\RA{{\bf RA}}
\def\RRA{{\bf RRA}}
\def\RCA{{\bf RCA}}

\def\G{{\bf G}}

\def\CRA{{\sf CRA}}

\def\F{{\mathfrak{F}}}
\def\At{{\sf{At}}}
\def\N{\mathbb{N}}
\def\R{\mathfrak{R}}

\def\CRA{{\sf CRA}}

\def\RPEA{{\sf RPEA}}

\def\cyl#1{{\sf c}_{#1}}

\def\c #1{{\cal #1}}

\def\pa{$\forall$}
\def\pe{$\exists$}

\def\nodes{{\sf nodes}}

\def\Ra{{\mathfrak{Ra}}}
\def\Nr{{\mathfrak{Nr}}}
\def\Z{{\cal Z}}
\def\CA{{\bf CA}}
\def\RCA{{\bf RCA}}
\def\c#1{{\mathcal #1}}

\def\set#1{ \{#1\}}

\def\Ca{{\mathfrak Ca}}

\def\pe{$\exists$}
\def\pa{$\forall$}
\def\Cm{{\mathfrak Cm}}
\def\Sg{{\mathfrak Sg}}
\def\P{{\mathfrak P}}
\def\Rl{{\mathfrak Rl}}
\def\N{{\cal N}}

\def\At{{\sf At}}

\def\rng{{\sf rng}}
\def\dom{{\sf dom}}
\def\Cm{{\sf Cm}}
\def\Mat{{\sf Mat}}
\def\w{{\sf w}}
\def\g{{\sf g}}
\def\y{{\sf y}}
\def\r{{\sf r}}

\def\cyl#1{{\sf c}_{#1}}

\def\ws{winning strategy}

 \def\CA{{\sf CA}}

\def\RCA{{\sf RCA}}

\def\y{{\sf y}}
\def\g{{\sf g}}
\def\r{{\sf r}}
\def\w{{\sf w}}

\def\Z{{\mathbb{Z}}}
\def\N{{\mathbb{N}}}

\def\RA{{\sf RA}}
\def\RRA{{\sf RRA}}
\def\EF{{\sf EF}}

\title{Blowing up and blurring Monk's algebras, and rainbow algebras}
\author{Tarek Sayed Ahmed \\
Department of Mathematics, Faculty of Science,\\
Cairo University, Giza, Egypt.
  }
\begin{document}
\maketitle
\begin{myabstract} We abstract an existing theme in the literature which Andr\'eka and N\'emeti call a a blow up and blur construction,
a very indicative term.
The idea is that one starts with a finite, hence atomic  algebra that is not representable. Then one
splits every atom of it into infinitely many to get a new atom structure.
The term algebra will be representable by a finite set of blurs, which are essentially non principal ultrafilters used as colours in the representation.
The complex algebra will not be representable because the the finite algebra  embeds into it
taking every atom to the join of its $\omega$ copies. This does not contradict the representability of the term
because these joins
do not exist in it, it is not complete (only finite and cofinite joins will exist in it ).
One can start with a Monk algebra or a  Rainbow algebra, and the potential of obtaining new results this
way is huge, as we show in the present paper.

We present several known examples in a general setting
and suggest new ones. We also  abstract a lifting  argument due to Monk
that enables one to transfer deep theorems in the finite dimensional case to infinite dimensions like the famous
problem 2.12 \cite{HMT1}, by applying this method to Hirsch-Hodkinson's
algebras, solving the finite dimensional case.(This idea was already implemented by the first author and Robin Hirsch in a submitted article
but in a narrower context).
We also give several sufficient very plausible conditions (concerning the existence
of a finite (possibly rainbow) relation or cylindric algebra),
that implies that classes of subneat reducts, namely, $S\Nr_n\CA_{n+k}$, $n\geq 3$ finite and $k\geq 4$
are not closed under Dedekind completions, and indeed, we provide and prove this result.
This is a long standing open problem that was, to the best of our knowledge, first explicitly
formulated in \cite{HHbook}.
\end{myabstract}

\section{Introduction}

We follow the notation of \cite{1} which is in conformity with that of \cite{HMT1}.
Assume that we have a class of Boolean algebras with operators for which we have a semantical notion of representability
(like Boolean set algebras or cylindric set algebras).
A weakly representable atom structure is an atom structure such that at least one atomic algebra based on it is representable.
It is strongly representable  if all atomic algebras having this atom structure are representable.
The former is equivalent to that the term algebra, that is, the algebra generated by the atoms,
in the complex algebra is representable, while the latter is equivalent  to that the complex algebra is representable.

Could an atom structure be possibly weakly representable but {\it not} strongly representable?
Ian Hodkinson, showed that this can indeed happen for both cylindric  algebras of finite dimension $\geq 3$, and relation algebras,
in the context of showing that the class of representable algebras, in both cases,  is not closed under completions.
Witnessed on  atomic algebras, it follows
the variety of representable algebras is not atom-canonical, for finite dimensions $>2$.(The complex algebra of an atom structure is
the completion of the term algebra.)
This construction was horribly complicated using a rainbow atom structure; it was simplified and streamlined,
by many authors, including us, but Hodkinson's construction, as we indicate below, has a very large potential to prove other theorems on completions
concerning subvarieties of
the representable algebras, and in fact, we realize this potential.

We start by presenting two distinct constructions for weakly representable atom structures that are not strongly representable.
We consider relation algebras and cylindric algebras. We will present these two constructions by blowing up a little
the  {\it blow up and blur} construction,
a very appropriate and suggestive term and construction invented by Andr'eka and N\'emeti.

The construction  we abstract here is to blow up a finite structure, replacing each 'colour or atom' by infinitely many, using blurs
to  represent the resulting term algebra, but the blurs are not enough to blur the structure of the finite structure in the complex algebra.
Then, the latter cannot be representable due to a {\it finite- infinite} contradiction.
This structure can be a finite clique in a graph or a finite relation algebra or a finite
cylindric algebra.

The main idea is to {\it split and blur}. Split  what? You can split a clique by taking $\omega$ many disjoint copies of it,
you can split  a finite relation algebra,
by splitting each atom into $\omega$ many, you can split a finite cylindric algebra. Generally, the splitting has to do with blowing up a finite structure
into infinitely many. And indeed, the splitting here has a lot of affinity with Andr\'eka's methods of splitting.

Then blur what? On this split one adds a subset of  a set of fixed in advance blurs, usually finite,
and then define an infinite atom structure, induced by the properties of the finite structure he originally started with.
It is not this atom structure that is blurred but rather the original finite structure.
This means that the term algebra built on this new atom structure,
that is the algebra generated by the atoms,
coincides with a carefully chosen partition of the set of atoms obtained after splitting and blurring
{\it up to minimal  deviations}, so the original finite relation algebra is blurred to the extent that is invisible on this level.

The term algebra will be representable, using all such blurs as colours,
But the original algebra structure re-appears in the completion of this term algebra, that is the complex algebra of the atom structure,
forcing it to be
non representable, due to a finite-infinite discrepancy. However, if the blurs are infinite, then, they will blur also the structure
of the small algebra in the complex algebra, and the latter could be representable, inducing a complete representation of
the term algebra.

We give two constructions, each one for both relation and cylindric algebras,
of weakly representable atom structures that are not strongly
representable, in the context of the blow up and blur construction.
In the first case we start with the cylindric algebra atom structure, which will be model theoretically defined
from a class of labelled graphs; the labels coming from a fixed in advance graph,
then extract from it a relation algebra atom structure, with an $n$ dimensional cylindric basis, such that algebra we started
with contains the term cylindric algebra generated by the basic matrices.
In the second case, we follow the more usual convention, we
start with the relation algebra with an $n$ dimensional cylindric bases; the required cylindric algebra atom structure
will be again the basic matrices. 
The second construction is due to Andr\'eka and N\'emeti, though we do not present the construction in its most general form.

These non- representable complex algebras each having a finite number of blurs converge in a precise
sense to a representable one, that in  an also precise sense, has infinitely many  blurs. The finite blurs viewed from the
graph theoretic point of view are colours, so such algebras can be viewed as based on graphs with finite chromatic number.

This is a typical Monk argument theme, non-representable algebras based on graphs with finite
chromatic number converging to one that is based on a graph having infinite chromatic number,
hence, representable. (The limit on the level of algebras is the ultraproduct, and that on the graphs it can be a disjoint union, or an ultraproduct as well).
It follows immediately that the variety of representable algebras is not finitely axiomatizable.

Monk's seminal result proved in 1969, showing that the class of representable cylindric algebras is not finitely axiomatizable had a
shatterring effect on  algebraic logic, in many respects. The conclusions drawn from this result, were that either the
extra non-Boolean basic operations of cylindrifiers and diagonal elements were not properly chosen, or that the notion of
representability was inappropriate; for sure it was concrete enough, but perhaps this is precisely the reason, it is far {\it too concrete.}

Research following both paths, either by changing the signature or/and altering the notion of concrete representability have been pursued
ever since, with amazing success.  Indeed there are  two conflicting but complementary facets
of such already extensive  research referred to in the literature, as 'attacking the representation problem'.
One is to delve deeply in investigating the complexity of potential axiomatizations for existing varieties
of representable algebras, the other is to try to sidestep
such wild unruly complex axiomatizations, often referred to as {\it taming methods}.

Those taming methods can either involve passing to (better behaved) expansions of the algebras considered,
or even completely change the signature  bearing in mind that the essential operations like cylindrifiers are
term definable or else change the very  notion of representability involved, as long as it remains concrete enough.

The borderlines are difficult to draw, we might not know what is {\it not} concrete enough, but we can
judge that a given representability notion is satisfactory, once we have one.

One can find well motivated appropriate notions of semantics by first locating them while giving up classical semantical prejudices.
It is hard to give a precise mathematical underpinning to such intuitions. What really counts at the end of the day
is a completeness theorem stating a natural fit between chosen intuitive concrete-enough, but perhaps not {\it excessively} concrete,
semantics and well behaved axiomatizations.
The move of altering semantics has radical philosophical repercussions, taking us away from the conventional
Tarskian semantics captured by Godel-like axiomatization; the latter completeness proof is effective
but highly undecidable; and this property is inherited by finite variable fragments
of first order logic as long as we insist on Tarskian semantics.

We have learnt from the history of development of algebraic logic that certain `undesirable'  properties follow from square
semantics. But undesirable depends on the point of view. To our mind, these results are not 'negative' at all.
On the contrary, they were proved using very sophisticated and versatile techniques, ranging from deep neat embeddings theorems to
the use of probabilistic Erdos graphs. These graphs were used to give anti-Monk ultraproducts,
that is algebras based on graphs with infinite chromatic number converging to one that is based
on a $2$ colourable graph. This amazing construction due to Hirsch and Hodkinson, proves that the class of strongly representable
atom structures is
not elementary.

This construction will be generalized, inspired by the ideas and constructions of Hirsch and Hodkinson, in the late \cite{1},
to give a Monk like atomic algebra based on an arbitrary graph $\G$,
constructed from the basic matrices of an atomic relation algebra also based on $\G$.
such that the weak and strong representability of both structures depend on the chromatic number of the graph $\G$.
Both are strongly representable or both are weakly representable. Using Erdos' probabilistic graphs, one can obtain the results of Hirsch and
Hodkinson that the class of strongly representable atom structures for both relation and cylindric algebras, is not elementary, but in one go.

One of our new result in this paper is that the class $S\Nr_n\CA_{n+4}$ is not closed under completions, in fact, it is not even elementary. 
The notion of neat reducts and neat embeddings
wil be prominent in our paper, as well. The reason is that we shall have occasion to deal with different types of representable algebras, 
like strongly representable algebras, weakly representable algebras (these are just atomic representable algebras), 
and completely reporesentable algebras. (These notions are obtained from the corresponding notions alraedy existing for atom structures in 
the obvious way). 

Representable algebras are just the algebras that have the neat embedding property, 
thus it is to be expected that such types of representable algebras
can be characterized via special need embeddings. Indeed, this has been already done for the completely representable ones, by the present author and 
Robin Hirsch independently..

For our results concerning neat reducts, we 
use techniques of Hirsch's in \cite{r} that deal with relation algebras,  and those of Hirsch and Hodkinson in \cite{hh} 
on complete representations. The results in the latter 
had to do with investigating the existence of complete representations for cylindric algebras  and for
this purpose, an infinite (atomic) game that tests complete representability was devised, and such a game was used
on a rainbow relation algebra.
The rainbow construction has a very wide scope of applications, 
and it  proved to be a nut cracker in solving many hard problems for relation algebras,
particularly for constructing counterexamples distinguishing between classes that are very subtly related, or rather unrelated. 

Unfortunately, relation rainbow algebras do not posses cylindric basis for $n\geq 4$ 
(so it seems that we cannot have our cake and eat it), so to
prove the analogous  result for cylindric algebra the construction had to be considerably modified to adapt the new situution, but the essence of
the two constructions is basically the same.
Instead of using atomic networks, in the cylindric algebra case games are played on {\it coloured graphs}. 
On the one hand, such graphs have edges which code the
relation algebra construction, but they also have hyperdges of length $n-1$, reflecing the cylindric algebra structure. 

It seems that there is no general theorem for rainbow constructions when it comes to cylindric like algebras, 
namely one relating winning strategies for pebble games on two structures or graphs $A,B$, to winning strategies for \pe\ in the cylindric rainbow 
algebra based $A$ and $B$, \cite{HHbook2}. However, below we show that in many 
concrete cases a \ws\ for either player in the pebble game  on two relational structures can be transferred to a \ws\ for 
the same player on the rainbow cylindric algebras, except that like the relation algebra case, the player needs two more pebbles 
(expressed by nodes of a graph) and one more round. 

Nevertheless, 
in the latebook on cylindric algebra \cite{1}, a general rainbow construction
is given in \cite{HHbook} in the context of building algebras from graphs giving rise to a class of models, 
from which the rainbow atom structure is defined, but just referring to one graph as a parameter, rather than two structures as done in their 
earlier book \cite{HHbook2}.
The second graph is fixed to be the greens; these are the set of colours that \pe\ never uses.
The class of models we deal with will be coloured graphs, viewed as structures for a natural signature (the rainbow signature where each colour is 
viewed as a relation symbol). 
To draw the analogy with relation algebras we treat the greens as an abstract 
structure, namely, an irreflexive complete graph. 

A \ws\ for \pe\ boils down to labeling edges that are appexes of  two cones ( a cone being a special coloured graph) 
inducing the same linear order on a face $F$ 
provided by \pa\ as part of his move. So here we have two nodes $\delta$ which is new and $\beta$.
This is the hardest part. 
Forced a red, \pe\ takes the red clique induced by all appexes of cones based on $F$, 
and she labels the new edge between $\delta$ and $\beta$ by a red, with double subscript, one  with the index of one of the nodes, namely $\beta$ 
(this is uniquely determined by the clique and
$\beta$),  the other suffix is the pebble $b$,  which is the pebble that  \pe\  responds to, in her private game to the pebble $a$
where $a$ is the tint of the cone with base $F$, and appex $\beta$. Here, \pe\ also plays the pebble 
$a$ playing the role of \pa\ ; her \ws\ in the private game allows her to do that. 
(This will be elaborated upon below).

\pa\ s strategy involves bombarding \pe\ with cones, whose tints are determined by his \ws\ in his private pebble game,
forcing a win on a red clique by forcing 
\pe\ to play an inconsistent red. 

For cylindric algebras, we take the $n$ neat reducts of algebras in higher dimension, ending up with a $\CA_n$, 
but we can also take {\it relation algebra reducts}, getting instead a relation algebra. 
The class of relation algebra reducts of cylindric algebras of dimension $n\geq 3$, denoted 
by $\Ra\CA_n$. The $\Ra$ reduct of a $\CA_n$, $\A$, is obtained by taking the $2$ neat reduct of $\A$, then defining composition and converse 
using one space dimension.
For $n\geq 4$, $\Ra\CA_n\subseteq \RA$. Robin Hirsch dealt primarily with this class in \cite{r}.   
This class has also been investigated by many authors, 
like Monk, Maddux, N\'emeti and Simon (A chapter in Simon's dissertation is devoted to such a class, 
when $n=3$). 
After a list of results and publications, Simon proved $\Ra\CA_3$ is not closed under subalgebras for $n=3$, 
with a persucor by Maddux proving the cases $n\geq 5$,
and Monk proving the case $n=4$.

In \cite{r}, Hirsch deals only the relation algebras proving that the $\Ra$ reducts of $\CA_k$s, $k\geq 5$, 
is not elementary, and he ignored the $\CA$ case, probably
because  of analogous results proved by the author on neat reducts \cite{IGPL}.

But  the results in these two last  papers are not identical (via a replacement of relation algebra via a cylindric algebra and vice versa).
There are differences and similarities that are illuminating 
for both, and the differences go both ways. 

For example in the $\RA$  case Hirsch proved that the elementary subalgebra that is not an $\Ra$ reduct 
is not a complete subalgebra of the one that is.
In the cylindric algebra case, the elementary subalgebra that is not a neat reduct 
constructed is a complete subalgebra of the neat reduct.

Hirsch \cite{r} also proved that any $\K$, such that $\Ra\CA_{\omega}\subseteq \K\subseteq S_c\Ra\CA_k$, $k\geq 5$ 
is not elementary;  here, using a rainbow construction for cylindric algebras, we prove its $\CA$ analogue.
In the same paper \cite{r}. In op.cit Robin asks whether the inclusion $\Ra\CA_n\subseteq S_c\Ra\CA_n$ is proper, the construction 
in \cite{IGPL}, shows that for $n$ neat reducts, 
it is. 

Besides giving a unified  proof of all cylindric like algebras for finite dimensions, 
we show that the inclusion is proper given that a certain $\CA_n$ term exists. 
(This is a usual first order formula using $n$ variables).
And indeed  using the technique in \cite{IGPL} we prove an analogous result for relation algebras,
answering the above  question of Hirsch's in \cite{r}. We show that there is an $\A\in \Ra\CA_{\omega}$ with a  an elmentary subalgebra
$\B\in S_c\Ra\CA_{\omega}$, that is not in $\Ra\CA_k$ when $\leq 5$. In particular, $\Ra\CA_k\subseteq S_c\Ra\CA_5$, for
$k\geq 5$.

Now it is worthwhile to reverse the deed, and generalize Hirsch's construction using rainbow cylindric algebras, to more results than that
obtained for cylindric algebras on neat reducts in \cite{IGPL}. 
For example, our construction here will give the following result not proved in {\it op.cit}:
There is an algebra $\A\in \Nr_n\CA_{\omega}$ with an elementary subalgebra, that is not completely
representable. But since the algebra $\A$ has countably many atoms, then it is completely representable.
This gives the result in \cite{hh}.

The transfer from results on relation algebras to 
cylindric algebras is not mechanical at all. More often than not, this is not an easy task, 
indeed it is far from being  trivial.

The layout of the paper is as follows.  In section one we present the blowing up and blurring of both Monk and Rainbow algebras,
giving a sufficient condition on the existence of certain finite relation algebras (having enough blurs), or a finite rainbow cylindric algebra
that implies that the class $S\Nr_n\CA_{n+k}$ $n\geq 3$ and $k\geq 4$ is not atom canonical.
In section 2 we abstract a lifting argument due to Monk implemented via ultraproducts, that enables one to transfer deep results proved
for the finite dimensional case, to  infinite dimension. As an example we solve problem  2.12 for cylindric algebras and polyadic algebras,
by lifting Monk-like finite dimensional algebras constructed by Hirsch and Hodkinson.
We discuss the possibility of lifting anti-Monk algebras to infinite dimensions, based on Erdos graphs, to
obtain that the class of strongly representable atom  structures is not elementary, even for infinite dimension.

Then we present several cylindric rainbow constructions. We prove that the class of completely
relativized $n$ square representable algebras is not elementary for $n\geq 5.$

Next, we concentrate on results concerning neat embeddings.We also construct our desired rainbow finite cylindric algebra, by constructing an algebra for which \pa\ has a \ws\ in a
certain atomic finite rounded game on coloured graphs using $n+4$ pebbles, implying that the
algebra is not in $S\Nr_n\CA_{n+4}$. The term algebra based on the same colours,
with the exception that every red is replaced by  $\omega$ many copies, will witness that the class $S\Nr_n\CA_{n+4}$ is not atom canonical, for
the finite algebra embeds into the completion of the blown up and blurred term algebra
by mapping every red to the join of its copies, and the term algebra is representable using a flexible
shade of red. Also it does not contain these joins, for it is not complete (it only contains finite or cofinite joins).
The consequences of this result are endless, we do not formulate them here, instead we refer to \cite{Hodkinson} for an all rounded picture
of this kind of results.
Then we prove that several classes related to the class of $n$ neat reducts  of $m$ dimensional cylindric algebras
for various $2<n<m$ are not elementary,  also using a rainbow construction
for cylindric algebras.

The paper will be divided into two complete wholes, the first deals with Monk-like algebras very occasionally intervened with a rainbow flavor,
the second deals with rainbow algebras only.

\section*{Part 1}
\bigskip
\section{Blow up and blur constructions}

We start by giving an abstract  definition of blowing up and blurring a finite structure. All our examples will fit this somewhat general
definition. In what follows, by an atom structure,
we mean an atom structure of any class of completely additive Boolean algebras.

Let $N$ be a finite structure, in our subsequent investigations $N$ will be further specified, it can be a clique, a relation algebra, a cylindric algebra,
any graph.  But there is no reason to impose any further restrictions on $N$ our next definition, which we try
keep as general as possible.
\begin{definition}
\begin{enumarab}

\item A {\it splitting} of $N$ is a product $N\times I$, where $I$ is an infinite set, so that it is forming $I$ copies of $N$.

\item $N$ is bad if $N$ is a graph that has a finite colouring, and so it follows that $N\times I$ is also bad.

\item$ N$ is good if $N$ is a graph  that has chromatic number $\infty$.

\item  A {\it blur} for $N$ is any  set  $J$.

\item $N$ is {\it blown up} if, there exists a set $J$ of blurs and a splitting $I$, and an atom structure
with underlying set $X=N\times I \times J$; the latter atom structure is called a blur of $N$ via $J$,
and is denoted by $\alpha(N,J).$  Furthermore, there is a one to one correspondence between $J$ and a a subset of non-principal
Boolean ultrafilters in  $\Cm X$,

\item $\alpha(N,J)$ is weak if $J$ is finite, in which case it is said to be weakly blurred,  else (if $J$ is infinite)
it is strong, in which case
we say that it is strongly blurred.

\item We formulate this for relation algebras. An atom structure $\alpha(N, J)$ {\it reflects}  a graph $N$, if the chromatic number of $N$
is coded in $\Cm\alpha(N,J)$.  By this we mean  that there is a $k=|J|$, possibly infinite, such that $N$
induces a partition $\P$ of $\Cm\alpha(N,J)$ into $N\times k$
sets, $P_0,\ldots P_{(N\times k)}$ that can be viewed as a partition of a coloured graph, namely $N\times I\times  k$,  into
independent sets (no edges between nodes), and each such set is monochromatic (its elements have the same colour).
(so that for all $P$ such that $P\in \P$, $(P;P).P=0$, composition will be defined in all cases to
allow all polychromatic triangles, and forbid independent monochromatic ones.)

\item A complex  algebra of an atom structure $\alpha(N, J)$ is good if a graph of infinite chromatic number is coded in it, otherwise it is bad.

\end{enumarab}
\end{definition}

Variants of the first and blow up construction we present now, appeared initially in \cite{Hodkinson},
which will largely influence our blow up and blur constructions, and in \cite{weak} and \cite{k}.
The last two references simplify Hodkinson's construction, one builds two relativized
set algebras based on a certain model that is in turn a Fraisse limit of a class of certain
class of labelled graphs, with the labels coming from $\G\cup \{\rho\}\times n$, where $\G$ is an arbitrary graph and $\rho$ is a new colour.
Under certain conditions on $\G$, the first set algebra based on $L_n$ will be representable, the second, its completion, based on $L_{\infty, n}$ is not.
Hodkinson's construction is as rainbow construction. The construction to be presented is Monk-like.

Let $\G$ be a graph. One can  define a family of coloured graphs $\cal F$ such that every edge of each graph $\Gamma\in {\cal F}$,
is coloured  by a unique label from
$\G\cup \{\rho\}\times n$, $\rho\notin \G$, in a carefully chosen way. The colour of $(\rho, i)$ is
defined to be $i$. The \textit{colour} of $(a, i)$ for $a \in \G$  is $i$.
$\cal F$ consists of all complete labelled graphs $\Gamma$ (possibly
the empty graph) such that for all distinct $ x, y, z \in \Gamma$,
writing $ (a, i) = \Gamma (y, x)$, $ (b, j) = \Gamma (y, z)$, $ (c,l) = \Gamma (x, z)$, we have:\\
\begin{enumarab}
\item $| \{ i, j, l \} > 1 $, or
\item $ a, b, c \in \G$ and $ \{ a, b, c \} $ has at least one edge
of $\G$, or
\item exactly one of $a, b, c$ -- say, $a$ -- is $\rho$, and $bc$ is
an edge of $\G$, or
\item two or more of $a, b, c$ are $\rho$.
\end{enumarab}

One forms a labelled graph $M$ which can be viewed as an $n$ homogeneous model of a natural signature,
namely, the one with relation symbols $(a, i)$, for each $a \in \G \cup \{\rho\}$, $i<n$.

Then one takes a subset $W\subseteq {}^nM$, by roughly dropping assignments whose edges not satisfy $(\rho, l)$ for every $l<n$.
Formally, $W = \{ \bar{a} \in {}^n M : M \models ( \bigwedge_{i < j < n,
l < n} \neg (\rho, l)(x_i, x_j))(\bar{a}) \}.$
Basically, we are throwing away assignments $\bar{a}$ whose edges between two of its elements are labelled by $\rho$, and keeping those
whose edges of its elements are not.
All this can be done with an arbitrary graph.

Now for particular choices of $\G$; for example if $\G$ is a certain rainbow graph, like Hodkinson's
or more simply a countable infinite collection of pairwise
union of disjoint $N$ cliques with $N\geq n(n-1)/2,$ \cite{k} or  is the graph whose nodes are the natural numbers, and the edge relation is defined by
$iEj$ iff $0<|i-j|<N$ \cite{weak}, for some finite $N$; here, the choice of $N$ is not haphazard,
but it a bound of edges of complete graphs having $n$ nodes, the
relativized set algebras based on $M$, but permitting as assignments satisfying formulas only $n$ sequences in $W$ will be an atomic
representable algebra.

This algebra, call it $\A$, has universe $\{\phi^M: \phi\in L^n\}$ where $\phi^M=\{s\in W: M\models \phi[s]\}.$ (This is not representable by its definition
because its unit is not a square.) Here $\phi^M$ denotes the permitted assignments satisfying $\phi$ in $M$.
Its completion is the relativized set algebra $\C$ which has universe the larger $\{\phi^M: \phi\in L^n_{\infty,\omega}\}$,
which turns out not representable. (All logics are taken in the above signature).
The isomorphism from  $\Cm\At\A$ to $\C$ is given by $X\mapsto \bigcup X$.

Let us formulate this construction in the context of split and blur.
Take the $n$ disjoint copies of $N\times \omega=\G$. Let $a\in \G\times n$. Then $a\in N\times \omega\times n$.
Then for every $(a,i)$ where $a\in N\times \omega$, and $i<n$, we have an atom $(a,i)^{\M}\in \A$.
The term algebra of $\A$ is generated by those.

Hence $N\times \omega\times n$ is the atom structure of $\A$ which can be weakly represented using the $n$ blurs, namely the set
$\{(\rho, i): i<n\}$.
The clique  $N$ is coded in the complex algebra level, forcing a finite $N$ colouring,
so that the complex algebra cannot be representable; a representation necessarily contradicts Ramsey's theorem.

We note that if $N$ is infinite, then the complex algebra (which is the completion of the algebra constructed
as above) will be representable and so $\A$, together the term algebra, will be  completely
representable.

$\Mat_n{\sf R}$ is the atom structure of basic matrices on ${\sf R}$.
From the above construction, we easily get:

\begin{theorem} Let $\G$ be a graph that is a disjoint union of cliques having size $N\geq n(n-1)/2$, $n\geq 3$.
Then there is a strongly $n$ homogeneous
labelled graph $M$ (when viewed as a model in a suitable signature consisting only of binary relation symbols),
every edge is labelled by an element from $\G\cup \{\rho\}\times n$,
$W\subseteq {}^nM$, such that the set algebra based on $W$ is an atomic $\A\in \RPEA_n$, and there is an atomic
$\R\in \RRA$, the latter with an $n$ dimensional polyadic basis,
such that $\A\cong {\sf Mat}_n\R$,  and the completions of the diagonal free reduct of $\A$, and $\R$ are not representable,
hence they are not completely
representable.
\end{theorem}
\begin{proof} \cite{weak} One defines a relation atom structure as follows.
We use the graph $N\times \omega$ of countably many disjoint $N$ cliques.
We define a relation algebra atom structure $\alpha(\G)$ of the form
$(\{1'\}\cup (\G\times n), R_{1'}, \breve{R}, R_;)$.
The only identity atom is $1'$. All atoms are self converse,
so $\breve{R}=\{(a, a): a \text { an atom }\}.$
The colour of an atom $(a,i)\in \G\times n$ is $i$. The identity $1'$ has no colour. A triple $(a,b,c)$
of atoms in $\alpha(\G)$ is consistent if
$R;(a,b,c)$ holds. Then the consistent triples are $(a,b,c)$ where

\begin{itemize}

\item one of $a,b,c$ is $1'$ and the other two are equal, or

\item none of $a,b,c$ is $1'$ and they do not all have the same colour, or

\item $a=(a', i), b=(b', i)$ and $c=(c', i)$ for some $i<n$ and
$a',b',c'\in \G$, and there exists at least one graph edge
of $G$ in $\{a', b', c'\}$.

\end{itemize}
$\alpha(\G)$ can be checked to be a relation atom structure. It is exactly the same as that used by Hirsch and Hodkinson in \cite{HHbook}, except
that we use $n$ colours, instead of just $3$, so that it a Monk algebra not a rainbow one. However, some monochromatic triangles
are allowed namely the dependent ones.
This allows the relation algebra to have an $n$ dimensional cylindric basis
and, in fact, the atom structure of $\A$ is isomorphic (as a cylindric algebra
atom structure) to the atom structure $\Mat_n$ of all $n$-dimensional basic
matrices over the relation algebra atom structure $\alpha(\G)$.
Indeed, for each  $m  \in {\cal M}_n, \,\ \textrm{let} \,\ \alpha_m
= \bigwedge_{i,j<n}  \alpha_{ij}. $ Here $ \alpha_{ij}$ is $x_i =
x_j$ if $ m_{ij} = 1$' and $R(x_i, x_j)$ otherwise, where $R =
m_{ij} \in L$. Then the map $(m \mapsto
\alpha^W_m)_{m \in {\cal M}_n}$ is a well - defined isomorphism of
$n$-dimensional cylindric algebra atom structures.
Let $M$ and $\A$ be as above. Then it is straightforward to define the polyadic operations on $\A$ by just
swapping variables in formulas.
So the set algebras based on $\A$ will be closed under the substitution operators, the former will be a polyadic set algebra.
the latter will be its completion, such that its its $df$ reduct is not representable since it is generated
by $<n$ dimensional elements.
It follows that $\Rd_{df}\A$ is not completely representable.
\end{proof}
\section{Reflections on the blow up and blur construction of Andr\'eka and N\'emeti}

Follows is theorem 1.1 in \cite{ANT}.

\begin{theorem}\label{OTT}
Suppose that $n$ is a finite ordinal with $n>2$ and $k\geq 0$.
There is a countable
representable
relation algebra ${\R}$
such that
\begin{enumroman}
\item Its completion, i.e. the complex algebra of its atom structure is
not representable, so $\R$ is representable but not completely representable
\item $\R$ is generated by a single element.
\item The (countable) set $\B_n{\R}$ of all $n$ by $n$ basic matrices over $\R$
constitutes an $n$-dimensional cylindric basis.
Thus $\B_n{\R}$ is a cylindric atom structure
and the full complex algebra $\Cm(\B_n{\R})$
with universe the power set of $\B_n{\R}$
is an $n$-dimensional cylindric algebra
\item The {\it term algebra} over the atom structure
$\B_n{\R}$, which is the countable subalgebra of $\Cm(\B_n{\R})$
generated by the countable set of
$n$ by $n$ basic matrices, $\C=\Tm(B_n \R)$ for short,
is a countable representable $\CA_n$, but $\Cm(\B_n)$ is not representable.
\item Hence $\C$ is a simple, atomic representable but not completely representable $\CA_n$
\item $\C$ is generated by a single $2$ dimensional element $g$, the relation algebraic reduct
of $\C$ does not have a complete representation and is also generated by $g$ as a relation algebra, and
$\C$ is a neat reduct of some simple representable $\D\in \CA_{n+k}$
such that the relation algebraic reducts of $\C$ and $\D$
coincide.
\end{enumroman}
\end{theorem}
\begin{proof} \cite{ANT}. Here we give an idea of the proof which is a blow up and blur construction.
For the technical details one is referred to the original paper \cite{ANT}, or to the sketch in \cite{SayedBook}.
Below we will return to this proof, and discuss its modifying to solve a long standing open problem in algebraic logic.
One starts with a finite Maddux relation algebra, that cannot be represented on finite sets.
Then this algebra is blown up and blurred. It is blown up by splitting the atoms each to infinitely many.
It is blurred by using a finite set of blurs or colours $J$. This can be expressed by the product $\At=\omega\times \At M\times J$,
which will define an infinite atom structure of a new
relation algebra. One can view such a product as a ternary matrix with $\omega$ rows, and for each fixed $n\in \omega$,  we have the rectangle
$\At M\times J$.
Then two partitions are defined on $\At$, call them $P$ and $E$.
Composition is defined on this new infinite atom structure; it is induced by the composition in $M$, and a tenary relation $e$
on $\omega$, that synchronizes which three rectangles sitting on the $i,j,k$ $e$ related rows compose like the original algebra $M$.
This relation is definable in the first order structure $(\omega, <)$.

The first partition $P$ is used to show that $M$ embeds in the complex algebra of this new atom structure, so
the latter cannot be represented, because it can only be represented on infinite sets.

The second partition divides $\At$ into $\omega$ sided finitely many rectangles, each with base $W\in J$,
and the the term algebra over $\At$, are the sets that intersect co-finitely with every member of this partition.
On the level of the term algebra $M$ is blurred, so that the embedding of the small algebra into
the complex algebra via taking infinite joins, do not exist in the term algebra for only finite and cofinite joins exist
in the term algebra.

The term algebra is representable using the finite number of blurs. These correspond to non-principal ultrafilters
in the Boolean reduct, which are necessary to
represent this term algebra, for the principal ultrafilter alone would give a complete representation,
hence a representation of the complex algebra and this is impossible.
Thereby an  an atom structure that is weakly representable but not strongly representable is obtained.

This atom structure has an $n$- dimensional cylindric basis, and so the $n$
basic matrices form an atom structure that is also only weakly representable.
The resulting $n$ dimensional cylindric term algebra obtained is a $k$ neat reduct that is not completely
representable. To make the algebra  one generated one uses Maddux's combinatorial techniques,
and this entails using infinitely many ternary relations.

\end{proof}
The construction above is very flexible, and the algebras constructed are based on Maddux-Monk algebras.

\begin{example}

There are several parameters used to define the relation algebra above.
Let $l\in \omega$, $l\geq 2$, and let $\mu$ be a non-zero cardinal. Let $I$ be a finite set,
$|I|\geq 3l.$ Let
$$J=\{(X,n): X\subseteq I, |X|=l,n<\mu\}.$$
$I$ is the atoms of $M$. $J$ is the set of blurs.

Pending on $l$ and $\mu$, let us call these atom structures ${\cal F}(l,\mu).$
In the example referred to above the atoms of $M$ are $I$, $J\subseteq \wp(I)$
consisting of $2$ element subsets of $I$
so it is just  ${\cal F}(2,1)$,

If $\mu\geq \omega$, then $J$ would be infinite,
and $Uf$ will be a proper subset of the ultrafilters.
It is not difficult to show that if $l\geq \omega$
(and we relax the condition that $I$ be finite), then
$\Cm{\cal F}(l,\mu)$ is completely representable,
and if $l<\omega$ then $\Cm{\cal F}(l,\mu)$ is not representable.
In the former case we have infnitely many colours, so that the chromatic number
of the graph is infinite, while in the second case the chromatic number is finite.

Informally, if the blurs get arbitrarily large, then in the limit, the resulting algebra will be completely representable, and so its complex algebra
will be representable. If we take, a sequence of blurs, each finite,
we get a sequence of Monk (non-respresentable) algebras whose
limit is completely representable
\end{example}
Formally
\begin{corollary}
\begin{enumarab}
\item  The classes ${\sf RRA}$ is not finitely axiomatizable.
\item  The elementary
closure of the class ${\sf CRA}$ is not finitely axiomatizable.
\end{enumarab}
\end{corollary}
\begin{demo}{Proof}
Let ${\cal D}$ be a non-
trivial ultraproduct of the atom structures ${\cal F}(i,1)$, $i\in \omega$. Then $\Cm{\cal D}$
is completely representable.
Thus $\Tm{\cal F}(i,1)$ are ${\sf RRA}$'s
without a complete representation while their ultraproduct has a complete representation.
Also $\Cm{\cal F}(i,1)$, $i\in \omega$
are non-representable with a completely representable ultraproduct.
\end{demo}

\subsection{Blowing up and blurring a finite rainbow relation algebra}

Hirsch and Hodkinson showed that there is an atomic relation algebra, that is representable, but the complex algebra of its atom structure
is not in $S\Ra \CA_6$. This algebra is obtained by blowing up and blurring a finite rainbow relation algebra namely $A_{K_4,K_3}$,
the idea is that \pa\ can win the $6$ rounded pebble game.
After splitting the reds, each into $\omega$ many copies, the algebra becomes representable, because
this basically produces a flexible non principal ultrafilter, namely the ultarfilter that intersects with reds with distinct indices co-finitely,
and this allows \pe\ to win an $\omega$ rounded
{\it non  atomic game} using this flexible ultrafilter to label the edges when he is forced a red. The complex algebra is not
in $S\Ra\CA_n$ because the finite relation algebra is embeddable into it by taking every red to its $\omega$ copies. The latter sets
do not exist in the term algebra, since  for any $X\in T$, $X$ intersects the two disjoint
sets of reds, namely, those with distinct indices, and those of equal indices
finitely or cofinitely, so that the set $\{\r_{il}^m: l, m\in K_3\}$, the image of $r_{il}$ cannot be in $T$.

The relation algebra obtained by Hirsch and Hodkinson  does not have an $n$ dimensional cylindric basis, except for $n=3$,
and using this it was proved  that the class $S\Nr_3\CA_k$ is not closed under completions, that is for the lowest value of $n$.

This construction, though, has a lot of affinity to the Andr\'eka N\'emeti blow up and blur construction (whose relation algebra has enough set of blurs to
allow a cylindric basis).
The major difference is that in the latter a Maddux finite algebra is used, while in the former case
a finite rainbow algebra is used. In the latter case, we can only infer that the complex algebra is non representable,
in the former case we can know and indeed we can prove more. The reason basically is  that non representability
of Maddux's algebras depends on an uncontrolable big Ramsey number (that is a function in the dimension),
while for  rainbow algebras we can control {\it when the algebra stops to be representable} by \pa\ s moves.
\pa\ forces a win by using greens, it is precisely this number, that  determines the extra dimensions in which the complex
algebras stop to be neatly embeddable
into, it is the point at which it outfits the reds.

{\it What can occur to ones mind here here is substitute a Maddux  finite relation algebra used by Andr\'eka and N\'emeti,
by the rainbow algebra mentioned used by Hirsch and Hodkinson and using the arguments of Andr\'eka and N\'emeti, we prove our stronger result.
A weaker version, can be gotten by lifting Hirsch and  Hodkinsons
construction from relation algebras whose atoms are colours to cylindric algebras
whose atoms are coloured graphs. In this paper, we perform the second task, solving a long standing open problem on completions of subneat reducts, and we  formulate the first task in the form of a conditional (i.e if then) theorem, in a while.}

We have two blown up and blurred relation algebras. But we want an $n$ dimensional cylindric algebra.
Using the notation in \cite{ANT}, given a relation algebra $\A$ with a set $J$ of blurs:

\begin{definition}
\begin{enumarab}
\item The blurs are adequate, if
$$(\forall V_1,\ldots V_n,W_2,\ldots W_n\in J)(\exists T\in J$$
$$(\forall 2\leq i\leq n)(\forall a\in V_i)(\forall b\in W_i(\forall c\in T)(a\leq b;c).$$
\item The blurs are strongly adequate if $\exists$ is replaced by $\forall$
\end{enumarab}
\end{definition}

\begin{theorem} If there  exists a relation algebra $\sf R$ that is not in $S\Ra\CA_{6}$, with an adequate set of $n+k$ blurs,
then there is an atomic infinite relation algebra ${\sf R}$ obtained by blowing up and bluring a
finite rainbow relation algebra; ${\sf R}$ has an $n+k$ dimensional cylindric basis
$\Tm{\sf Mat}_{n+k} {\sf R}\in \RCA_{n+k}$, $\Tm{\sf Mat}_n{\sf R}\cong \Nr_m\Tm{\sf Mat}_{n+k}{\sf R}$
 and $\Cm{\sf Mat_n{\sf R}}$ is
not in $S\Nr_n\CA_{n+k}$.
Furthermore, this cylindric algebra, witnessing that $S\Nr_n\CA_{n+k}$ is not closed under completions,
can be chosen to be generated by a single element.
\end{theorem}
\begin{proof} Let $k\geq 3$ and $m\in \omega$. That the term algebra is in $\RCA_n\cap \Nr_n\CA_{n+m}$ is exactly
like in \cite{ANT}. Now the
complex algebra $\Cm\F$ is embeddable into the $\Ra$ reduct of he complex algebra $\Cm\Mat_n$. So if $\Cm\Mat_n$ is in
$S\Nr_n\CA_{n+6}$, then $\Cm\F\in \Ra S\Nr_n\CA_{n+6}\subseteq \Ra\CA_6.$
contradiction, hence we are done.
\end{proof}

\begin{theorem}\label{hodkinson} Assume that there exists a finite
rainbow cylindric algebra, not in $S\Nr_n\CA_{n+k}$. Then the class $\SNr_n\CA_{n+k}$ is not closed under completions.
\end{theorem}

\begin{proof} For each red $r_{ij}$ in the colours of the finite algebra, take the new reds to be $r_{ij}^l$, $l\in \omega$ ($\omega$ copies of $r_{ij}$),
together with a  shade of red $\rho$.
Then the resulting atom structure will be exactly like Hodkinson's in \cite{Hodkinson}, except
that we have only finitely many greens (the same number of greens in  the finite algebra; we only split the reds).
Let $M$ be the model constructed as in \cite{Hodkinson} from the new colours, as a limit of coloured graphs. Then $M$ is an
$n$ homogeneous model of the rainbow signature ($\rho$ is outside the the signature, though it occurs as a label for coloured graphs),
satisfying the  $L_{\omega_1, \omega}$ rainbow theory,
which actually consists of only first order formulas, because we have finitely
many greens.
Furthermore, if one takes the set algebra  $\A$ with universe $\{\phi^M: \phi\in L_n\},$ where $\phi^M$ is the set of all assignments
having no edge labelled by $\rho$, then it will contain the term algebra,
its atoms will be essentially the coloured graphs having no edges coloured by
$\rho$,  and its completion will be $\{\phi^M: \phi\in L^n_{\infty}\}$ restricting to the same assignments.
As before the finite algebra embeds into the latter by $r_{ij}\mapsto \bigvee_{l\in \omega} [r_{ij}^l](x_0, x_{n-1})^W$, where $W$ is the set
of above assignments. Therefore the class $\SNr_n\CA_{n+k}$ is not closed under completions, for their term algebra will be representable,
but its completion will not be in $S\Nr_n\CA_{n+k}$.
\end{proof}
Later, we will show that such an algebra exists, by a rainbow construction,  solving a long standing open problem reported `officially' as of $2002$,
in \cite{HHbook}, and re appearing in \cite{1}.

The next theorem says that our first model theoretic  and second Andr\'eka-Ne\'meti blow up and blur constructions actually fit our general framework.
\begin{theorem}
\begin{enumarab}
\item For the second construction, $N$ is the clique of size $n(n-1)/2$, the splitting is $N\times \omega\times n$, the blurs are
$\{(\rho, i): i<n\}$, the split up and blurred atom structure is $\At\A$. In the first case, the structure is $\bold M$,
the splitting is $\At \bold M\times I\times J$, where $J$ is the set of blurs, and the blurred up structure is ${\cal R}$.

\item In the  first case $N$ can be coded in $\Cm\At\A$,
and in the second case {\it the graph} on which $\bold M$ is based (avoiding monochromatic triangles), is coded in $\Cm\At\R$.
This prohibits the representability of the latter, because any such representation, will give a monochromatic
triangle; this is a typical Monk argument.

\item Both structures are only weakly blurred via a finite set of colours, hence they have finite chromatic number,
and so any representation induces a finite partition of the complex
algebra, and then Ramsey's theorem enforces a dependent monochromatic triangle.
\end{enumarab}
\end{theorem}
We should mention at this point that there could be no countable atomic algebra in $\Nr_n\CA_{\omega}$ 
that is not strongly representable. For any such algebra will be necessarily completely representable, and hence
strongly representable. In this sense the result of Andr\'eka and N\'emeti stated 
above is the best possible. Arbitrarily large $k$ cannot be replaced by 
$omega$.

The underlying idea here is to choose a graph $\Gamma$, such that the Monk structures or 
Rainbow structures based on  $\Gamma$ has an $n$ homogeneous countable model that has quantifier 
elimination. 
This model will encode all $n$ coloured graphs (structures), namely the atoms, 
and the set algebra based on this graph (obtained by dropping assignments labelled by  one or more flexible ultrafilter or refexive node),
will be  representable. The term algebra will be representable, precisely because it is {\it not} complete, so precarious joins are not there,
only finite or cofinite ones are. 
But its completion, the complex algebra of its atom structure, will not be 
representable, because for the precise reason it {\it is complete}, and precarious joins will deliver an inconsistency,  prohibiting 
a representation, by forcing an infinite finite discrepancy or an inconsistent triangle. 

Blow up and blur construction, applied to both Mon-like and rainbow algebras apt  for such a task.

\section{Lifting Monk--Hirsch-Hodkinson algebras to infinite dimensions}

It is not the case that every algebra in $\CA_m$ is the neat reduct of an algebra in $\CA_n$,
nor need it even be a subalgebra of a neat reduct of an algebra in $\CA_n$.  Furthermore, $S\Nr_m\CA_{m+k+1}\neq S\Nr_m\CA_m$,
whenever $3\leq m<\omega$ and $k<\omega$.

The hypothesis in the following theorem presupposes the existence of certain finite
dimensional algebras, not chosen haphazardly at all, but are  rather an abstraction of cylindric algebras existing in the literature witnessing the last
proper inclusions.  The main idea, that leads to the conclusion of the theorem,
is to use such finite dimensional algebras to obtain an an analogous result for the infinite dimensional case.
Accordingly, we found it convenient streamline Monk's argument who did exactly that for cylindric algebras, but we do it in
the wider context of systems of varieties of Boolean algebras with operators definable by a schema. (Strictly speaking Monk's lifting argument is weaker,
the infinite dimensional constructed algebras are merely non -representable, in our case they are not only non-representable,
but are also subneat reducts of  algebras in a given pre
assigned dimension; this is a technical difference, that needs some non-trivial fine turning in the proof).
The inclusion of finite dimensions in our formulation, was therefore not a luxuary, nor was it motivated by
aesthetic reasons, and nor was it merely an artefect of Monk's definition.
It is motivated by the academic worthiness of the result
(for infinite dimensions).

\begin{theorem}\label{2.12} Let $(\K_{\alpha}: \alpha\geq 2)$ be a complete system of varieties definable by a schema.
Assume that for $3\leq m<n<\omega$,
there is $m$ dimensional  algebra $\C(m,n,r)$ such that
\begin{enumarab}
\item $\C(m,n,r)\in S\Nr_m\K_n$
\item $\C(m,n,r)\notin S\Nr_m\K_{n+1}$
\item $\prod_{r\in \omega} \C(m, n,r)\in S\Nr_m\K_n$
\item For $m<n$ and $k\geq 1$, there exists $x_n\in \C(n,n+k,r)$ such that $\C(m,m+k,r)\cong \Rl_{x}\C(n, n+k, r).$
\end{enumarab}
Then for any ordinal $\alpha\geq \omega$, $S\Nr_{\alpha}\K_{\alpha+k+1}$ is not axiomatizable by a finite schema over $S\Nr_{\alpha}\K_{\alpha+k}$
\end{theorem}

\begin{proof} The proof is a lifting argument essentially due to Monk, by 'stretching' dimensions  using only properties of
reducts  and ultraproducts, formalizable in the context of a system of varieties definable by a schema.
This method was recently used by Hirsch and Sayed Ahmed to show that for any class of algebras $K$
between Pinter's substitution algebras and polyadic equality algebras,
the class $\S\Nr_n\K_{n+k+1}$ is not finitely axiomatizable over $S\Nr_n\K_{n+k}$ using Monk like algebras.
This works for all dimensions.
The proof is divided into 3 parts:

\begin{enumarab}

\item  Let $\alpha$ be an infinite ordinal,
let $X$ be any finite subset of $\alpha$, let $I=\set{\Gamma:X\subseteq\Gamma\subseteq\alpha,\; |\Gamma|<\omega}$.
For each $\Gamma\in I$ let $M_\Gamma=\set{\Delta\in I:\Delta\supseteq\Gamma}$ and let $F$ be any ultrafilter over $I$
such that for all $\Gamma\in I$ we have $M_\Gamma\in F$
(such an ultrafilter exists because $M_{\Gamma_1}\cap M_{\Gamma_2} = M_{\Gamma_1\cup\Gamma_2}$).
For each $\Gamma\in I$ let $\rho_\Gamma$ be a bijection from $|\Gamma|$ onto $\Gamma$.
For each $\Gamma\in I$ let $\c A_\Gamma, \c B_\Gamma$ be $\K_\alpha$-type algebras.
If for each $\Gamma\in I$ we have
$\Rd^{\rho_\Gamma}\c A_\Gamma=\Rd^{\rho_\Gamma}\c B_\Gamma$ then $\Pi_{\Gamma/F}\c A_\Gamma=\Pi_{\Gamma/F}\c B_\Gamma$.
Standard proof, by Los' theorem.
Note that the base of $\Pi_{\Gamma/F}\c A_\Gamma$ is
identical with the base of $\Pi_{\Gamma/F}\Rd^{\rho_\Gamma}\c A_\rho$
which is identical with the base of $\Pi_{\Gamma/F}\c B_\Gamma$, by the assumption in the lemma.
Each operator $o$ of $\K_\alpha$ is the same for both ultraproducts because $\set{\Gamma\in I:\dim(o)\subseteq\rng(\rho_\Gamma)} \in F$.

Furthermore, if $\Rd^{\rho_\Gamma}\c A_\Gamma \in \K_{|\Gamma|}$, for each $\Gamma\in I$  then $\Pi_{\Gamma/F}\c A_\Gamma\in \K_\alpha$.
For this, it suffices to prove that each of the defining axioms for $\K_\alpha$ holds for $\Pi_{\Gamma/F}\c A_\Gamma$.
Let $\sigma=\tau$ be one of the defining equations for $\K_{\alpha}$, the number of dimension variables is finite, say $n$.
Take any $i_0, i_1,\ldots  i_{n-1}\in\alpha$, we must prove that
$\Pi_{\Gamma/F}\c A_\Gamma\models \sigma(i_0,\ldots i_{n-1})=\tau(i_0\ldots  i_{n-1})$.
If they are all in $\rng(\rho_\Gamma)$, say $i_0=\rho_\Gamma(j_0), \; i_1=\rho_\Gamma(j_1), \;\ldots i_{n-1}=\rho_\Gamma(j_{n-1})$,
then $\Rd^{\rho_\Gamma}\c A_\Gamma\models \sigma(j_0, \ldots ,j_{n-1})=\tau(j_0, \ldots j_{n-1})$,
since $\Rd^{\rho_\Gamma}\c A_\Gamma\in\K_{|\Gamma|}$, so $\c A_\Gamma\models\sigma(i_0\ldots , i_{n-1})=\tau(i_0\ldots i_{n-1}$.
Hence $\set{\Gamma\in I:\c A_\Gamma\models\sigma(i_0, \ldots, i_{n-1}l)=\tau(i_0, \ldots,  i_{n-1})}\supseteq\set{\Gamma\in I:i_0,\ldots,  i_{n-1}
\in\rng(\rho_\Gamma}\in F$,
hence $\Pi_{\Gamma/F}\c A_\Gamma\models\sigma(i_0,\ldots  i_{n-1})=\tau(i_0, \ldots,  i_{n-1})$.
Thus $\Pi_{\Gamma/F}\c A_\Gamma\in\K_\alpha$.

\item Let $k\in \omega$. Let $\alpha$ be an infinite ordinal.
Then $S\Nr_{\alpha}\K_{\alpha+k+1}\subset S\Nr_{\alpha}\K_{\alpha+k}.$
Let $r\in \omega$.
Let $I=\{\Gamma: \Gamma\subseteq \alpha,  |\Gamma|<\omega\}$.
For each $\Gamma\in I$, let $M_{\Gamma}=\{\Delta\in I: \Gamma\subseteq \Delta\}$,
and let $F$ be an ultrafilter on $I$ such that $\forall\Gamma\in I,\; M_{\Gamma}\in F$.
For each $\Gamma\in I$, let $\rho_{\Gamma}$
be a one to one function from $|\Gamma|$ onto $\Gamma.$
Let ${\c C}_{\Gamma}^r$ be an algebra similar to $\K_{\alpha}$ such that
\[\Rd^{\rho_\Gamma}{\c C}_{\Gamma}^r={\c C}(|\Gamma|, |\Gamma|+k,r).\]
Let
\[\B^r=\prod_{\Gamma/F\in I}\c C_{\Gamma}^r.\]
We will prove that
\begin{enumerate}
\item\label{en:1} $\B^r\in S\Nr_\alpha\K_{\alpha+k}$ and
\item\label{en:2} $\B^r\not\in S\Nr_\alpha\K_{\alpha+k+1}$.  \end{enumerate}

The theorem will follow, since $\Rd_\K\B^r\in S\Nr_\alpha \K_{\alpha+k} \setminus S\Nr_\alpha\K_{\alpha+k+1}$.

For the first part, for each $\Gamma\in I$ we know that $\c C(|\Gamma|+k, |\Gamma|+k, r) \in\K_{|\Gamma|+k}$ and
$\Nr_{|\Gamma|}\c C(|\Gamma|+k, |\Gamma|+k, r)\cong\c C(|\Gamma|, |\Gamma|+k, r)$.
Let $\sigma_{\Gamma}$ be a one to one function
 $(|\Gamma|+k)\rightarrow(\alpha+k)$ such that $\rho_{\Gamma}\subseteq \sigma_{\Gamma}$
and $\sigma_{\Gamma}(|\Gamma|+i)=\alpha+i$ for every $i<k$. Let $\c A_{\Gamma}$ be an algebra similar to a
$\K_{\alpha+k}$ such that
$\Rd^{\sigma_\Gamma}\c A_{\Gamma}=\c C(|\Gamma|+k, |\Gamma|+k, r)$.  By the second part
with  $\alpha+k$ in place of $\alpha$,\/ $m\cup \set{\alpha+i:i<k}$
in place of $X$,\/ $\set{\Gamma\subseteq \alpha+k: |\Gamma|<\omega,\;  X\subseteq\Gamma}$
in place of $I$, and with $\sigma_\Gamma$ in place of $\rho_\Gamma$, we know that  $\Pi_{\Gamma/F}\A_{\Gamma}\in \K_{\alpha+k}$.

We prove that $\B^r\subseteq \Nr_\alpha\Pi_{\Gamma/F}\c A_\Gamma$.  Recall that $\B^r=\Pi_{\Gamma/F}\c C^r_\Gamma$ and note
that $C^r_{\Gamma}\subseteq A_{\Gamma}$
(the base of $C^r_\Gamma$ is $\c C(|\Gamma|, |\Gamma|+k, r)$, the base of $A_\Gamma$ is $\c C(|\Gamma|+k, |\Gamma|+k, r)$).
 So, for each $\Gamma\in I$,
\begin{align*}
\Rd^{\rho_{\Gamma}}\C_{\Gamma}^r&=\c C((|\Gamma|, |\Gamma|+k, r)\\
&\cong\Nr_{|\Gamma|}\c C(|\Gamma|+k, |\Gamma|+k, r)\\
&=\Nr_{|\Gamma|}\Rd^{\sigma_{\Gamma}}\A_{\Gamma}\\
&=\Rd^{\sigma_\Gamma}\Nr_\Gamma\A_\Gamma\\
&=\Rd^{\rho_\Gamma}\Nr_\Gamma\A_\Gamma
\end{align*}
By the first part of the first part we deduce that
$\Pi_{\Gamma/F}\C^r_\Gamma\cong\Pi_{\Gamma/F}\Nr_\Gamma\A_\Gamma\subseteq\Nr_\alpha\Pi_{\Gamma/F}\A_\Gamma$,
proving \eqref{en:1}.

Now we prove
Now we prove \eqref{en:2}.
For this assume, seeking a contradiction, that $\B^r\in S\Nr_{\alpha}\K_{\alpha+k+1}$,
$\B^r\subseteq \Nr_{\alpha}\c C$, where  $\c C\in \K_{\alpha+k+1}$.
Let $3\leq m<\omega$ and  $\lambda:m+k+1\rightarrow \alpha +k+1$ be the function defined by $\lambda(i)=i$ for $i<m$
and $\lambda(m+i)=\alpha+i$ for $i<k+1$.
Then $\Rd^\lambda(\c C)\in \K_{m+k+1}$ and $\Rd_m\B^r\subseteq \Nr_m\Rd^\lambda(\c C)$.
For each $\Gamma\in I$,\/  let $I_{|\Gamma|}$ be an isomorphism
\[{\c C}(m,m+k,r)\cong \Rl_{x_{|\Gamma|}}\Rd_m {\c C}(|\Gamma|, |\Gamma+k|,r).\]
Let $x=(x_{|\Gamma|}:\Gamma)/F$ and let $\iota( b)=(I_{|\Gamma|}b: \Gamma)/F$ for  $b\in \c C(m,m+k,r)$.
Then $\iota$ is an isomorphism from $\c C(m, m+k,r)$ into $\Rl_x\Rd_m\B^r$.
Then $\Rl_x\Rd_{m}\B^r\in S\Nr_m\K_{m+k+1}$.
It follows that  $\c C (m,m+k,r)\in S\Nr_{m}\K_{m+k+1}$ which is a contradiction and we are done.
\end{enumarab}
\end{proof}

\subsection{Monk's original algebras}

Monk defined the required algebras, witnessing the non finite axiomatizability of $\RCA_n$ $n\geq 3$,
via their atom structure. An $n$ dimensional atom structure is a triple
$\G=(G, T_i, E_{ij})_{i,j\in n}$
such that $T_i\subseteq G\times G$ and $E_{ij}\subseteq G$, for all $i, j\in n$. An atom structure so defined, is a cylindric atom structure if
its complex algebra $\Ca\G\in \CA_n$. $\Ca\C$ is the algebra
$$(\wp(G), \cap, \sim T_i^*, E_{ij}^*)_{i,j\in n},$$ where
$$T_i^*(X)=\{a\in G: \exists b\in X: (a,b)\in T_i\}$$
and
$$E_{i,j}^*=E_{i,j}.$$
Cylindric algebras are axiomatized by so-called Sahlqvist equations, and therefore it is easy to spell out first order correspondants
to such equations characterizing
atom structures of cylindric algebras.

\begin{definition}
For $3 \leq m\leq n < \omega$, ${\G}_{m, n}$ denotes
the cylindric atom structure such that ${\G}_{m, n} = (G_{m, n},
T_i, E_{i,j})_{i, j < m} $ of dimension
$m$ which is defined as follows:
$G_{m, n}$ consists of all pairs $(R, f)$ satisfying
the following conditions:
\begin{enumarab}
\item $R$ is equivalence relation on $m$,
\item $f$ maps $\{ (\kappa, \lambda) : \kappa, \lambda < n,
\kappa  \not{R} \lambda\}$ into $n$,
\item for all $\kappa, \lambda < m$, if $\kappa \not{R} \lambda$
then $f_{\kappa \lambda } = f_{\lambda \kappa}$,
\item for all $\kappa, \lambda, \mu < m$, if $\kappa \not{R}
\lambda R \mu$ then $f_{\kappa \lambda } = f_{\kappa \mu}$,
\item for all $\kappa, \lambda, \mu < n$, if $\kappa \not{R}
\lambda \not{R} \mu \not{R} \kappa$ then $|f_{\kappa \lambda },
f_{\kappa \mu}, f_{\lambda \mu}| \neq 1.$
\end{enumarab}
For $\kappa < m$ and $(R, f), (S, g) \in G(m,n)$ we define
\begin{eqnarray*}
&(R, f) T_\kappa (S, g) ~~ \textrm{iff} ~~ R \cap {}^2(n
\smallsetminus
\{\kappa\}) = S \cap {}^2(m \smallsetminus \{\kappa\}) \\
& \textrm{and for all} ~~ \lambda, \mu \in m \smallsetminus
\{\kappa\}, ~~ \textrm{if} ~~ \lambda \not{R} \mu~~ \textrm{then} ~~
f_{\lambda \mu } = g_{\lambda \mu}.
\end{eqnarray*}
For any $ \kappa, \lambda <m$, set
$$ E_{\kappa \lambda} = \{ ( R, f) \in G(m,n) : \kappa R \lambda \}.$$
\end{definition}
Monk proves that this indeed defines a cylindric atom structure, he defines
the $m$ dimensional cylindric algebra $\C(m,n)=\Ca(\G(m,n),$ then he proves:

\begin{theorem}
\begin{enumarab}
\item For $3\leq m\leq n<\omega$ and $n-1\leq \mu< \omega$, $\Nr_m\C(n,\mu)\cong \C(m,\mu)$.
In particular, $\C(m, m+k)\cong \Nr_m(\C(n, n+k)$.
\item Let $x_n=\{(R,f)\in G_{n, n+k}; R=(R\cap ^2n) \cup (Id\upharpoonright {}^2(n\sim m))\\
\text { for all $u, v$,} uRv, f(u,v)\in n+k,
\text { and
for all } \mu\in n\sim m, v<\mu,\\ f(\mu, v)=\mu+k\}.$

Then
$\C(n, n+k)\cong \Rl_x\Rd_n\C(m, m+k).$
\end{enumarab}
\end{theorem}
\begin{demo}{Proof} \cite{HMT2}, theorems 3.2.77 and 3.2.86.
\end{demo}
\begin{theorem} The class $\RCA_{\alpha}$ is not axiomatized by a finite schema.
\end{theorem}
\begin{demo}{Proof} By $\RCA_{\alpha}=S\Nr_{\alpha}\CA_{\alpha+\omega}.$ Let $r\in \omega$. Then $\B^r$, call it $\B_k$ constructed above,
from the finite dimensional algebras increasing in dimension, is in
$S\Nr_{\alpha}\CA_{\alpha+k}$ but it is not in $S\Nr_{\alpha}\CA_{\alpha+k+1}$ least representable.
Then the ultraproduct of the $\B_k$'s over a non-principal ultrafilter will be in
$S\Nr_{\alpha}\CA_{\alpha+\omega},$ hence will be representable.
\end{demo}

Johnsson defined a polyadic atom structure based on the $\G_{m,n}$. First a helpful piece of notation:
For relations $R$ and $G$, $R\circ G$ is the relation
$$\{(a,b): \exists c (a,c)\in R, (c, b)\in S\}.$$
Now Johnson extended the atom structure $\G(m,n)$ by

$(R,f)\equiv_{ij}(S,g)$ iff $f(i,j)=g(j,i)$ and if $(i,j)\in R$, then $R=S$, if not, then $R=S\circ [i,j]$, as composition of relations.

Strictly speaking, Johnsson did not define substitutions quite in this way; because he has all finite transformations, not only transpositions.
Then, quasipolyadic algebras was not formulated in schematizable form, a task accomplished by Sain and Thompson much later.

\begin{theorem}(Sain-Thompson) ${\sf RQA}_{\alpha}$ and ${\sf RQEA}_{\alpha}$ are not finite schema axiomatizable
\end{theorem}
\begin{demo}{Proof} One proof uses the fact that ${\sf RQA}_{\alpha}=S\Nr_{\alpha}{\sf QA}_{\alpha+\omega}$, and that the diagonal free reduct
Monk's algebras (hence their infinite dilations) are not representable. Another proof uses a result of Robin Hirsch and Tarek Sayed
Ahmed that there exists finite dimensional quasipolyadic algebras satisfying the hypothesis of theorem \ref{2.12}.
A completely analogous result holds for Pinters algebras, using also finite dimensional Pinters algebras satisfying the hypothesis
of theorem \ref{2.12}.
\end{demo}

\subsection{The Good and the Bad}

We recall the following definition from Hirsch and Hodkinson \cite{HHbook2}, except that we turn the glass around, we replace the model theoretic
definition based on structures satisfying
a set of first order sentences, by  coloured graphs which are precisely the models of these sentences. This makes the
affinity with Monk's original algebras easier to
discern.

\begin{definition} A Monk's algebra is an algebra defined as follows.
Let $\Gamma$ be a graph and and $\Gamma\times n$ be its $n$ disjoint copies.
We define a class $K$ of hypergraphs,
where only $n-1$ tuples are
coloured from $\Gamma\times n$.
$M$ is in $I(\Gamma)$, if
\begin{enumarab}
\item  Every $\bar{s}\in {}^{n-1}M$ and all its permutations have a unique colour.
\item if $(a_0,\ldots a_{i-1},\ldots a_{i+1})$ is labelled by $p_i$, for $i<n-1$, then there are $i<j$ such that $p_i$ $p_j$ has an edge.
\end{enumarab}
\end{definition}

Let $\rho(I(\Gamma))=\{f: f: n\to M\in K\}$ be the corresponding atom structure.
and $\M(\Gamma)=\Cm\rho(I(\Gamma))$.

\begin{theorem}  $\M(\Gamma)$ is representable if and only if $\Gamma$ has infinite chromatic number
\end{theorem}

\begin{proof}\cite{HHbook2} for the details. Here we will be sketchy.
Let $\Gamma^* $ be the ultrafilter extension of $\Gamma$.
We first define a strong bounded morphism $\Theta$
form $\M(\Gamma)_+$ to $\rho(I(\Gamma^*))$, as follows:
For any $x_0,\ldots x_{n-2}<n$ and $X\subseteq \Gamma^*\times n$, define the following element
of $\M(\Gamma^*)$:
$$X^{(x_0,\ldots x_{n-2})}=\{[f]\in \rho(I(\Gamma^*): \exists p\in X[M_f\models p(f(x_0),\ldots f(x_{n-2})]\}.$$

Let $\mu$ be an ultrafilter in $\M(\Gamma)$ Define $\sim $on $n$ by $i\sim j$ iff $d_{ij}\in \mu$ Let $g$ be the projection map from $n$ to $n/\sim$.

Define a $\Gamma^*\times n$ coloured graph by with domain $n/\sim$ as follows . For each $v\in \Gamma^*\times n$
and $x_0,\ldots x_{n-2}<n$, we let
$$M_{\mu}\models v(g(x_0),\ldots g(x_{n-2})\Longleftrightarrow  X^{(x_0,\ldots x_{n-2})}\in \mu.$$

We show that $\Cm(\M(\Gamma)_+)=\M(\Gamma)^{\sigma}$ is completely representable, by showing that
\pe\ has a \ws\ in the $\omega$ rounded atomic game on networks, by identifying networks with
structures.  Let $\Theta$ be as defined above. Let $N$ be a $\M(\Gamma)^{\sigma}$ network.

Then $\theta(N)$ is an $\M(\Gamma)^{\sigma}$ network. Identify $\Theta(N)$ with a structure $N^*$ with same domain and such that for
$x_0,\ldots x_{n-1}\in N$ with $\theta(N(x_0,\ldots x{n-1})=[f]$, say, each $i<n$ and each $p\in \Gamma^*\times n$,
we have $N^*\models p(x_0,\ldots x_{i-1}, x_{i+1},\ldots x_{n-1})$ if
$M_f\models p(f(0),\ldots f(i-1), f(i+1),\ldots f(n-1))$.

Assume that \pa\ chose a node $x\in {}^nN$, $i<n$ and an atom $[f]$ with $[f]\leq c_iN(x)$.
Assume that $f(i)\neq f(j)$ else she would have chosen the same network.
Let $y=x[i|z]\in {}^nN\cup \{z\}$, and $Y=\{y_0,\ldots y_{n-1}\}.$ Define
$q_j\in \Gamma\times n$ as follows. If $\bar{y}\sim y_i$
are pairwise distinct, let $q_j\in \Gamma\times n$ be the unique element satisfying
$M\models q_j (y_0\ldots y{n-1}, y_{j+1}, y_{n-1})$. Else we chose $q_j$ arbitrarily,
then chose a new copy and let $d$ be the relexive node in this copy.
Define $M\models d(t_0,\ldots t_{n-2})$ whenever $t_0,\ldots t_{n-2}\in M$
are distinct and $z\in \{t_0,\ldots t_{n-2}\}\nsubseteq Y$. This structure is as required, and we are done.

The converse follows from the fact that a representation of an algebra based on a graph with finite chromatic number, necessarily contradicts
Ramsey's theorem.
\end{proof}

\begin{definition}
\begin{enumarab}

\item A Monks algebra is good if $\chi(\Gamma)=\infty$

\item A Monk's algebra is bad if $\chi(\Gamma)<\infty$

\end{enumarab}
\end{definition}
It is easy to construct a good Monks algebra
as an ultraproduct (limit) of bad Monk algebras, as we did above, taking algebras having finitely many blurs converging to one with infinitely many,
more than enough to represent it.
The converse is hard.
It took Erdos probabilistic graphs, to get a sequence of good graphs converging to a bad one.

\begin{theorem}(Hirsch Hodkinson) The class of strongly representable atom structures, for any signature between ${\sf Df}$ and ${\sf PEA}$
of finite dimensions $n\geq 3$ is not elementary
\end{theorem}
\begin{proof}\cite{HHbook2}
\end{proof}
Let $\G$ be a graph. One can  define a family of first order structures in the signature $\G\cup \{\rho\}\times n$
as follows:
\begin{enumarab}
\item
For all $a,b\in M$, there is a unique $p\in \G\cup \{\rho\}\times n$, such that
$(a,b)\in p$.

\item If  $M\models (a,i)(x,y)\land (b,j)(y,z)\land (c,l)(x,z)$, then $| \{ i, j, l \}> 1 $, or
$ a, b, c \in \G$ and $\{ a, b, c\} $ has at least one edge
of $\G$, or exactly one of $a, b, c$ -- say, $a$ -- is $\rho$, and $bc$ is
an edge of $\G$, or two or more of $a, b, c$ are $\rho$.
\end{enumarab}
The second condition is exactly forbidding monochromatic independent triangles.
This can be  coded as a first order theory $T$.

The above construction is a kind of a Monk's algebras.
We will view the matter somewhat more deeply, inspired by ideas and constructions due to Hirsch and
Hodkinson.

We throw away the shade of red $\rho$, it will be recovered in the ultrafilter extension of $\G$.
More precisely, Let $\G$ be a graph. One can  define a family of first order structures in the signature $\G\times n$, denote it by $I(\G)$
as follows:
For all $a,b\in M$, there is a unique $p\in \G\times n$, such that
$(a,b)\in p$. If  $M\models (a,i)(x,y)\land (b,j)(y,z)\land (c,l)(x,z)$, then $| \{ i, j, l \}> 1 $, or
$ a, b, c \in \G$ and $\{ a, b, c\} $ has at least one edge
of $\G$.
For any graph $\Gamma$, let $\rho(\Gamma)$ be the atom structure defined from the class of models satisfying the above,
these are maps from $n\to M$, $M\in I(\G)$, endowed with an obvious equivalence relation,
with cylindrifiers and diagonal elements defined as \cite{HHbook2}, and let
$\M(\Gamma)$ be the complex algebra of this atom structure.

Fix $\G$, and let  $\G^* $ be the ultrafilter extension of $\G$.
We first define a strong bounded morphism $\Theta$
form $\M(\G)_+$ to $\rho(I(\G^*))$, as follows:
For any $x_0, x_1<n$ and $X\subseteq \G*\times n$, define the following element
of $\M(\G^*)$:
$$X^{(x_0, x_1)}=\{[f]\in \rho(I(\G^*)): \exists p\in X[M_f\models p(f(x_0),f(x_1))]\}.$$
Let $\mu$ be an ultrafilter in $\M(\Gamma)$. Define $\sim $on $n$ by $i\sim j$ iff $d_{ij}\in \mu$.
Let $g$ be the projection map from $n$ to $n/\sim$.
Define a $\G^*\times n$ coloured graph with domain $n/\sim$ as follows.
For each $v\in \Gamma^*\times n$
and $x_0, x_1<n$, we let
$$M_{\mu}\models v(g(x_0)\dots g(x_1))\Longleftrightarrow  X^{(x_0, x_1)}\in \mu.$$
Hence, any ultrafilter $\mu\in \M(\G)$ defines $M_{\mu}$ which is  a $\G^*$ structure.
If $\Gamma$ has infinite chromatic number, then $\G^*$ has a reflexive node, and this can be used
to completely represent $\M(\G))^{\sigma}$, hence represent  $\M(\G)$ as follows:
We show that \pe\ has a \ws\ in the $\omega$ rounded atomic game on networks.
Let $N$ be a given $\rho\M(\Gamma)$ network. Let $z\notin N$ and let $y=x[i|z]\in {}^n(N\cup \{z\}$. Let $d\in \G^*$ be a reflexive node, and
define $M\models d(t_0, t_1)$ if $z\in \{t_0, t_1\}\nsubseteq Y$, otherwise labelling the edges are like $N$.
One defines a relation atom structure as follows.
We use the graph $N\times \omega$ of countably many disjoint $N$ cliques.

We define a relation algebra atom structure $\alpha(\G)$ of the form
$(\{1'\}\cup (\G\times n), R_{1'}, \breve{R}, R_;)$.
The only identity atom is $1'$. All atoms are self converse,
so $\breve{R}=\{(a, a): a \text { an atom }\}.$
The colour of an atom $(a,i)\in \G\times n$ is $i$. The identity $1'$ has no colour. A triple $(a,b,c)$
of atoms in $\alpha(\G)$ is consistent if
$R;(a,b,c)$ holds. Then the consistent triples are $(a,b,c)$ where

\begin{itemize}

\item one of $a,b,c$ is $1'$ and the other two are equal, or

\item none of $a,b,c$ is $1'$ and they do not all have the same colour, or

\item $a=(a', i), b=(b', i)$ and $c=(c', i)$ for some $i<n$ and
$a',b',c'\in \G$, and there exists at least one graph edge
of $G$ in $\{a', b', c'\}$.

\end{itemize}

$\alpha(\G)$ can be checked to be a relation atom structure. It is exactly the same as that used by Hirsch and Hodkinson in \cite{HHbook}, except
that we use $n$ colours, instead of just $3$, so that it a Monk algebra not a rainbow one. However, some monochromatic triangles
are allowed namely the dependent ones.
This allows the relation algebra to have an $n$ dimensional cylindric basis
and, in fact, the atom structure of $\M(\Gamma)$ is isomorphic (as a cylindric algebra
atom structure) to the atom structure $\Mat_n$ of all $n$-dimensional basic
matrices over the relation algebra atom structure $\alpha(\G)$.

This is a variation on the construction of Hirsch and Hodkinson in \cite{HHbook2},
however there are two essential differences to the credit of this construction.
One is that the signature consists only of binary relation symbols, the other which actually is very much related to the previous condition,
is that it allows defining a relation algebra atom structure such that one
is weakly (strongly) representable if and only if the other is (strongly) weakly representable,
for they are based on the same graph $\G$ and so the chromatic number coded in the complex algebra of each is one and the same.

So in one go, we get:

\begin{theorem} The class of strongly representable atom structures of both relation and cylindric algebras of dimension $\geq 3$ 
is not elementary
\end{theorem}
\begin{proof} Using Erdos' graphs
\end{proof}

In what follows we make an attempt to lift the finite dimensional construction to the infinite dimensional case.
In our lifting argument implemented via ultraproducts
above, we had for every dimension, a sequence of bad algebras converging to a
good one in the same dimension.
Using ultraproducts we were able to construct a sequence of infinite dimensional bad algebras,
converging to a good one. Can we possibly reverse the process, here as well,
for infinite dimensions:

\begin{definition} A Monk's algebra is an algebra defined as follows.
Let $\Gamma$ be a graph and and $\bigcup_n\Gamma\times n$ be  the disjoint union of its $n$ disjoint copies.
We define a class $K=I(\Gamma)$ of hypergraphs,
where $m$ tuples, $m$ finite (could be arbitrarily large), are
coloured from $\bigcup_n\Gamma\times n$.
$M$ is in $I(\Gamma)$, if for all $p\in {}^{\omega}M$,
\begin{enumarab}
\item  Every $\bar{s}\in {}^{\omega}M^{(p)}$ and all its permutations have a unique colour.
\item  If $\bar{s}\in {}^{\omega}M^{(p)}$ has support $\{a_0,\ldots a_m\}$, and $(a_0,\ldots a_{i-1},\ldots a_{i+1}, )$ is labelled by $p_i$, for
$i<m-1$, then there are $i<j$ such that $p_i$ $p_j$ has an edge.
\end{enumarab}
\end{definition}

Let $\F$ is defined by taking $\bigcup\{^{\omega}M^{(p)}: p\in {}^{\omega}M \text { $M$  is a coloured graph }\}$,
and $\F/\sim$ is defined like in \cite{HHbook}, $f\in \F$ with the $\omega/\sim$ where
$\sim =ker f=\{(x,y): f(x)=f(y)\}$, and cylindrifiers are defined by $[f] T_i [g] \text { iff }
f\upharpoonright \omega\sim \{i\}=g\upharpoonright \omega\sim \{\i\}$, and
$[f]\in d_{ij}$ if $f(i)=f(j)$.
This can be easily checked to be a cylindric atom structure.

\begin{theorem}Suppose that $\M(\Gamma)$ is representable iff $\Gamma$ has infinite chromatic number
Then the class of weakly representable atom structures of $\omega$ dimensional cylindric algebras, is not elementary
\end{theorem}
\begin{demo} {Proof} Let $\Delta$ and $\Delta_k$ be as in \cite{HHbook2}, that is,  the graphs based on Erdos graphs. $\Delta_k$
has an infinite chromatic number, and its ultraproduct
$\Delta$ is $2$ colourable.
Fix $k$. For each $n\in \omega$, let $\C_n^k$ be an algebra of type $\CA_{\omega}$
such that $\Rd_n\C_n^k=\M_n(\Delta_k)$. Let $\A_k=\prod_n \C_n^k=\A(\Delta_k)$.
Then for each $k\in \omega$,
$\A(\Delta_k)$ is atomic and strongly representable.
But $\prod _k\A_k=\A(\Delta)$ is not.

\end{demo}

Call an atomic  representable algebra, strongly representable if $\Cm\At\A$ is representable, and call it weakly representable if it is just representable.
Let ${\sf SRCA_n}$ denotes the class of strongly representable algebras and ${\sf WRCA_n}$ the class of atomic representable algebras.
Then the later is elementary, the former is not. This prompts:

\begin{definition} An atomic $\RCA_n$ is bad if $\At\A$ is not strongly representable. Otherwise it is good.
\end{definition}.

\subsection{Monk- like  algebras again of Hirsch and Hodkinson}

Now we prove the conclusion of theorem \ref{2.12}, for cylindric algebras and quasipolyadic equality,
solving the infinite dimensional version of the famous 2.12 problem in algebraic logic.
The finite dimensional algebras we use are constructed by Hirsch and Hodkinson; and
they based on a relation algebra construction.
Such combinatorial algebras have affinity with Monk's algebras.
Related algebras were constructed by Robin Hirsch and the present author (together with the above lifting argument).

We recall the construction of Hirsch and Hodkinson. They prove their result for cylindric algebras.
Here, by noting that their atom structures are also symmetric; it permits expansion by substitutions,
we slightly extend the result to polyadic equality algebras.
Define  relation algebras $\A(n,r)$ having two parameters $n$ and $r$ with $3\leq n<\omega$ and $r<\omega$.
Let $\Psi$ satisfy $n,r\leq \Psi<\omega$. We specify the atom structure of $\A(n,r)$.
\begin{itemize}
\item The atoms of $\A(n,r)$ are $id$ and $a^k(i,j)$ for each $i<n-1$, $j<r$ and $k<\psi$.
\item All atoms are self converse.
\item We can list the forbidden triples $(a,b,c)$ of atoms of $\A(n,r)$- those such that
$a.(b;c)=0$. Those triples that are not forbidden are the consistent ones. This defines composition: for $x,y\in A(n,r)$ we have
$$x;y=\{a\in At(\A(n,r)); \exists b,c\in At\A: b\leq x, c\leq y, (a,b,c) \text { is consistent }\}$$
Now all permutations of the triple $(Id, s,t)$ will be inconsistent unless $t=s$.
Also, all permutations of the following triples are inconsistent:
$$(a^k(i,j), a^{k'}(i,j), a^{k''}(i,j')),$$
if $j\leq j'<r$ and $i<n-1$ and $k,k', k''<\Psi$.
All other triples are consistent.
\end{itemize}

Hirsch and Hodkinson invented means to pass from relation algebras to $n$ dimensional cylindric algebras,
when the relation algebras in question have what they call a hyperbasis.

Unless otherwise specified,
$\A=(A,+,\cdot, -,  0,1,\breve{} , ;, Id)$
will denote an arbitrary relation algeba with $\breve{}$ standing for converse, and $;$ standing for composition, and
$Id$ standing for the identity relation.

\begin{definition} Let $3\leq m\leq n\leq k<\omega$, and let $\Lambda$ be a non-empty set.
An $n$ wide $m$ dimensional $\Lambda$ hypernetwork over $\A$ is a map
$N:{}^{\leq n}m\to \Lambda\cup At\A$ such that $N(\bar{x})\in At\A$ if $|\bar{x}|=2$ and $N(\bar{x})\in \Lambda$ if $|\bar{x}|\neq 2$,
with the following properties:
\begin{itemize}
\item $N(x,x)\leq Id$ ( that is $N(\bar{x})\leq Id$ where $\bar{x}=(x,x)\in {}^2n.)$

\item $N(x,y)\leq N(x,z);N(z,y)$
for all $x,y,z<m$
\item If $\bar{x}, \bar{y}\in {}^{\leq n}m$, $|\bar{x}|=|\bar{y}|$ and $N(x_i,y_i)\leq Id$ for all $i<|\bar{x}|$, then $N(\bar{x})=N(\bar{y})$
\item when $n=m$, then $N$ is called an $n$ dimensional $\Lambda$ hypernetwork.

\end{itemize}
\end{definition}

\begin{definition} Let $M,N$ be $n$ wide $m$ dimensional $\Lambda$ hypernetworks.
\begin{enumarab}
\item For $x<m$ we write $M\equiv_xN$ if $M(\bar{y})=N(\bar{y})$ for all $\bar{y}\in {}^{\leq n}(m\sim \{x\})$
\item More generally, if $x_0,\ldots x_{k-1}<m$ we write $M\equiv_{x_0,\ldots,x_{k-1}}N$
if $M(\bar{y})=N(\bar{y})$ for all $\bar{y}\in {}^{\leq n}(m\sim \{x_0,\ldots x_{k-1}\}).$
\item If $N$ is an $n$ wide $m$ dimensional $\Lambda$ -hypernetwork over $\A$, and $\tau:m\to m$ is any map, then
$N\circ \tau$ denotes the $n$ wide $m$ dimensional $\Lambda$ hypernetwork over $\A$ with labellings defined by
$$N\circ \tau(\bar{x})=N(\tau(\bar{x})) \text { for all }\bar{x}\in {}^{\leq n}m.$$
That is
$$N\circ \tau(\bar{x})=N(\tau(x_0),\ldots ,\tau(x_{l-1}))$$
\end{enumarab}
\end{definition}

\begin{lemma} Let $N$ be an $n$ dimensional $\Lambda$ hypernetwork over $\A$ and $\tau:n\to n$ be a map.
Then $N\circ \tau$ is also a network.
\end{lemma}
\begin{demo}{Proof} \cite{HHbook} lemma 12.7
\end{demo}

\begin{definition} The set of all $n$ wise $m$ dimensional hypernetworks will be denoted by $H_m^n(\A,\Lambda)$.
An $n$ wide $m$ dimensional $\Lambda$
hyperbasis for $\A$ is a set $H\subseteq H_m^n(\A,\lambda)$ with the following properties:
\begin{itemize}
\item For all $a\in At\A$, there is an $N\in R$ such that $N(0,1)=a$
\item For all $N\in R$ all $x,y,z<n$ with $z\neq x,y$ and for all $a,b\in At\A$ such that
$N(x,y)\leq a;b$ there is $M\in R$ with $M\equiv_zN, M(x,z)=a$ and $M(z,y)=b$
\item For all $M,N\in H$ and $x,y<n$, with $M\equiv_{xy}N$, there is $L\in H$ such that
$M\equiv_xL\equiv_yN$
\item For a $k$ wide $n$ dimensional hypernetwork $N$, we let $N|_m^k$ the restriction of the map $N$ to $^{\leq k}m$.
For $H\subseteq H_n^k(\A,\lambda)$ we let $H|_k^m=\{N|_m^k: N\in H\}$.
\item When $n=m$, $H_n(\A,\Lambda)$ is called an $n$ dimensional hyperbases.
\end{itemize}
We say that $H$ is symmetric, if whenever $N\in H$ and $\sigma:m\to m$, then $N\circ\sigma\in H$.
\end{definition}
We note that $n$ dimensional hyperbasis are extensions of Maddux's notion of cylindric basis.

\begin{theorem} If $H$ is a $m$ wide $n$ dimensional $\Lambda$ symmetric
hyperbases for $\A$, then $\Ca H\in \PEA_n$.
\end{theorem}
\begin{demo}{Proof} Let $H$ be the set of $m$ wide $n$ dimensional $\Lambda$ symmetric hypernetworks for $\A$.
The domain of $\Ca(H)$ is $\wp(H)$.
The Boolean operations are defined as expected (as complement and union of sets). For $i,j<n$ the diagonal is defined by
$${\sf d}_{ij}=\{N\in H: N(i,j)\leq Id\}$$
and for $i<n$ we define the cylindrifier ${\sf c}_i$ by
$${\sf c}_iS=\{N\in H: \exists M\in S(N\equiv_i M\}.$$
Now the polyadic operations are defined by
$${\sf p}_{ij}X=\{N\in H: \exists M\in S(N=M\circ [i,j])\}$$

Then $\Ca(H)\in \PEA_n$. Furthermore, $\A$ embeds into
$\Ra(\Ca(H))$ via
$a\mapsto \{N\in H: N(0,1)\leq a\}.$
\end{demo}

\begin{theorem} Let $3\leq m\leq n\leq k<\omega$ be given.
Then $\Ca(H|^k_m)\cong \Nr_m(\Ca(H))$
\end{theorem}
\begin{demo}{Proof}\cite{HHbook} 12.22
\end{demo}

The set $C=H_n^{n+1}(\A(n,r), \Lambda)$ aff all $(n+1)$ wide $n$
dimensional $\Lambda$ hypernetworks over $\A(n,r)$ is an $n+1$ wide $n$
dimensional {\it symmetric} $\Lambda$ hyperbasis.
$H$ is symmetric, if whenever $N\in H$ and $\sigma:m\to m$, then $N\circ\sigma\in H$.
Hence $\A(n,r)$ embeds into the $\Ra$ reduct of $\C$.

\begin{theorem}
Assume that $3\leq m\leq n$, and let
$$\C(m,n,r)=\Ca(H_m^{n+1}(\A(n,r),  \omega)).$$ Then the following hold:
\begin{enumarab}
\item For any $r$ and $3\leq m\leq n<\omega$, we
have $\C(m,n,r)\in \Nr_m\PEA_n$.
\item  For $m<n$ and $k\geq 1$, there exists $x_n\in \C(n,n+k,r)$ such that $\C(m,m+k,r)\cong \Rl_{x}C(n, n+k, r).$
\item $S\Nr_{\alpha}\CA_{\alpha+k+1}$ is not axiomatizable by a finite schema over $S\Nr_{\alpha}\CA_{\alpha+k}$
\end{enumarab}
\end{theorem}
\begin{demo}{Proof}
\begin{enumarab}
\item $H_n^{n+1}(\A(n,r), \omega)$ is a wide $n$ dimensional $\omega$ symmetric hyperbases, so $\Ca H\in \PEA_n.$
But $H_m^{n+1}(\A(n,r),\omega)=H|_m^{n+1}$.
Thus
$$\C_r=\Ca(H_m^{n+1}(\A(n,r), \omega))=\Ca(H|_m^{n+1})\cong \Nr_m\Ca H$$
\item For $m<n$, let $$x_n=\{f\in F(n,n+k,r): m\leq j<n\to \exists i<m f(i,j)=Id\}.$$
Then $x_n\in \c C(n,n+k,r)$ and ${\sf c}_ix_n\cdot {\sf c}_jx_n=x_n$ for distinct $i, j<m$.
Furthermore
\[{I_n:\c C}(m,m+k,r)\cong \Rl_{x_n}\Rd_m {\c C}(n,n+k, r).\]
via
\[ I_n(S)=\{f\in F(n, n+k, r): f\upharpoonright m\times m\in S, \forall j(m\leq j<n\to  \exists i<m\; f(i,j)=Id)\}.\]

\item Follows from theorem \ref{2.12}.
\end{enumarab}
\end{demo}

\section*{PART 2}

\section{Various rainbow constructions for cylindric algebras}

Let $A$, $B$ be two relational structures. Let $\CA_{A, B}$ be the cylindric atom of coloured graphs.
That is its atom structure is based on the colours:

\begin{itemize}

\item greens: $\g_i$ ($1\leq i<n-2)$, $\g_0^i$, $i\in A$.

\item whites : $\w, \w_i: i<n-2$
\item yellow : $\y$
\item reds:  $\r_{ij}$ $(i,j\in B)$,

\item shades of yellow : $\y_S: S\subseteq_{\omega}B$, $S=B.$

\end{itemize}
And coloured graphs are:
\begin{definition}
\begin{enumarab}

\item $\Gamma$ is a complete graph.

\item $\Gamma$ contains no triangles (called forbidden triples)
of the following types:

\vspace{-.2in}
\begin{eqnarray}
&&\nonumber\\
(\g, \g^{'}, \g^{*}), (\g_i, \g_{i}, \w),
&&\mbox{any }i\in n-1\;  \\
(\g^j_0, \y, \w_i)&&\\
(\g^j_0, \g^k_0, \w_0)&&\mbox{ any } j, k\in A\\
\label{forb:pim}(\g^i_0, \g^j_0, \r_{kl})&&\\
\label{forb: black}(\y,\y,\y), \\
\label{forb:match}(\r_{ij}, \r_{j'k'}, \r_{i^*k^*})&&\mbox{unless }i=i^*,\; j=j'\mbox{ and }k'=k^*
\end{eqnarray}
and no other triple of atoms is forbidden.

\item If $a_0,\ldots   a_{n-2}\in \Gamma$ are distinct, and no edge $(a_i, a_j)$ $i<j<n$
is coloured green, then the sequence $(a_0, \ldots a_{n-2})$
is coloured a unique shade of yellow.
No other $(n-1)$ tuples are coloured shades of yellow.

\item If $D=\set{d_0,\ldots  d_{n-2}, \delta}\subseteq \Gamma$ and
$\Gamma\upharpoonright D$ is an $i$ cone with apex $\delta$, inducing the order
$d_0,\ldots  d_{n-2}$ on its base, and the tuple
$(d_0,\ldots d_{n-2})$ is coloured by a unique shade
$y_S$ then $i\in S.$

\end{enumarab}
\end{definition}
This is the class of structures $K$ we are dealing with, every element $M$ in is a coloured graph.
and the defining relations above can be coded in first order logic, more precisely,
every green, white,  red, atom corresponds to a binary relation, and every $n-1$ colour is coded as an $n-1$ relations,
and the colured graphs are defined
as the first order structures, of a set of $L_{\omega_1,\omega}$ as presented in \cite{HHbook2}.

Now from  these coloured graphs we define an atom structure of a $\CA_n$.
Let $$K=\{a: a \text { is a surjective map from $n$ onto some } \Gamma\in \bold J
\text { with nodes } \Gamma\subseteq \omega\}.$$
We write $\Gamma_a$ for the element of $K$ for which
$a:n\to \Gamma$ is a surjection.
Let $a, b\in K$ define the following equivalence relation: $a \sim b$ if and only if
\begin{itemize}
\item $a(i)=a(j)\text { and } b(i)=b(j)$

\item $\Gamma_a(a(i), a(j))=\Gamma_b(b(i), b(j))$ whenever defined

\item $\Gamma_a(a(k_0)\dots a(k_{n-2}))=\Gamma_b(b(k_0)\ldots b(k_{n-1}))$ whenever
defined
\end{itemize}
Let $\mathfrak{C}$ be the set of equivalences classes. Then define
$$[a]\in E_{ij} \text { iff } a(i)=a(j)$$
$$[a]T_i[b] \text { iff }a\upharpoonright n\sim \{i\}=b\upharpoonright n\sim \{i\}.$$

This defines a  $\CA_n$
atom structure.

Games on these atom structures are the atomic games played on networks \cite{HHbook}, \cite{hh}.
We translate them to games on graphs.
\begin{definition}
Let $\Gamma\in \bold J$ be arbitrary. Define the corresponding network $N_{\Gamma}$
on $\C_n$,
whose nodes are those of $\Gamma$
as follows. For each $a_0,\ldots a_{n-1}\in \Gamma$, define
$N_{\Gamma}(a_0,\ldots  a_{n-1})=[\alpha]$
where $\alpha: n\to \Gamma\upharpoonright \set{a_0,\ldots a_{n-1}}$ is given by
$\alpha(i)=a_i$ for all $i<n$. Then, as easily checked,  $N_{\Gamma}$ is an atomic $\C_{n}$ network.
Conversely, let $N$ be any non empty atomic $\C_n$ network.
Define a complete coloured graph $\Gamma_N$
whose nodes are the nodes of $N$ as follows:
\begin{itemize}
\item For all distinct $x,y\in \Gamma_N$ and edge colours $\eta$, $\Gamma_N(x,y)=\eta$
if and only if  for some $\bar z\in ^nN$, $i,j<n$, and atom $[\alpha]$, we have
$N(\bar z)=[\alpha]$, $z_i=x$ $z_j=y$ and the edge $(\alpha(i), \alpha(j))$
is coloured $\eta$ in the graph $\alpha$.

\item For all $x_0,\ldots x_{n-2}\in {}^{n-1}\Gamma_N$ and all yellows $\y_S$,
$\Gamma_N(x_0,\ldots x_{n-2})= \y_S$ if and only if
for some $\bar z$ in $^nN$, $i_0, \ldots  i_{n-2}<n$
and some atom $[\alpha]$, we have
$N(\bar z)=[\alpha]$, $z_{i_j}=x_j$ for each $j<n-1$ and the $n-1$ tuple
$\langle \alpha(i_0),\ldots \alpha(i_{n-2})\rangle$ is coloured
$\y_S.$ Then $\Gamma_N$ is well defined and is in $\bold J$.
\end{itemize}
\end{definition}
The following is then, though tedious and long,  easy to  check:

\begin{theorem}
For any $\Gamma\in \bold J$, we have  $\Gamma_{N_{\Gamma}}=\Gamma$,
and for any $\C_n$ network
$N$, $N_{{\Gamma}_N}=N.$
\end{theorem}
This translation makes the following equivalent formulation of the
the graphs games  originally defined on networks.

\begin{definition}
The new  game builds a nested sequence $\Gamma_0\subseteq \Gamma_1\subseteq \ldots $
of coloured graphs.
\pa\ picks a graph $\Gamma_0\in \bold J$ with $\Gamma_0$
Here the nodes of he graph are contained in $m$. \pe\  makes no response
to this move. In a subsequent round, let the last graph built be $\Gamma_i$.
$\forall$ picks
\begin{itemize}
\item a graph $\Phi\in \bold J$ with $|\Phi|=n$
\item a single node $k\in \Phi$
\item a coloured graph embedding $\theta:\Phi\sim \{k\}\to \Gamma_i$
Let $F=\phi\smallsetminus \{k\}$. Then $F$ is called a face.
\pe\ must respond by amalgamating
$\Gamma_i$ and $\Phi$ with the embedding $\theta$. In other words she has to define a
graph $\Gamma_{i+1}\in C$ and embeddings $\lambda:\Gamma_i\to \Gamma_{i+1}$
$\mu:\phi \to \Gamma_{i+1}$, such that $\lambda\circ \theta=\mu\upharpoonright F.$
\end{itemize}
\end{definition}

Now let us consider the possibilities. There may be already a point $z\in \Gamma_i$ such that
the map $(k\mapsto z)$ is an isomorphism over $F$.
In this case \pe\ does not need to extend the
graph $\Gamma_i$, she can simply let $\Gamma_{i+1}=\Gamma_i$
$\lambda=Id_{\Gamma_i}$, and $\mu\upharpoonright F=Id_F$, $\mu(\alpha)=z$.
Otherwise, without loss of generality,
let $F\subseteq \Gamma_i$, $k\notin \Gamma_i$.
Let ${\Gamma_i}^*$ be the colored graph with nodes $\nodes(\Gamma_i)\cup\{k\}$,
whose edges are the combined edges of $\Gamma_i$ and $\Phi$,
such that for any $n-1$ tuple $\bar x$ of nodes of
${\Gamma_i}^*$, the color ${\Gamma_i}^*(\bar x)$ is
\begin{itemize}
\item $\Gamma_i(\bar x)$ if the nodes of $x$
all lie in $\Gamma$ and $\Gamma_i(\bar x)$ is defined
\item $\phi(\bar x)$ if the nodes of $\bar x$ all lie in
$\phi$ and $\phi(\bar x)$ is defined
\item undefined, otherwise.
\end{itemize}
\pe\ has to complete the labeling of $\Gamma_i^*$ by adding
all missing edges, colouring each edge $(\beta, k)$
for $\beta\in \Gamma_i\sim\Phi$ and then choosing a shade of
white for every $n-1$ tuple $\bar a$
of distinct elements of ${\Gamma_i}^*$
not wholly contained in $\Gamma_i$ nor $\Phi$,
if non of the edges in $\bar a$ is coloured green.
She must do this on such a way that the resulting graph belongs to $\bold J$.
If she survives each round, \pe\ has won the play

\begin{definition} A red clique is a coloured graph all of whose edges are red. The index of a node $n$ in a red clique is defined by
$\mu(n)=b\in B$ where $\Gamma(n, m)=\r_{bb'}$, for some $m\in \Gamma$ and $b'\in B$. This is well defined.
\end{definition}
In this part of the paper we have a lot of varying parameters, different dimensions of algebras, number of pebbles in a game,
number of nodes in a graph,  number of rounds in a game, so
to avoid confusion we will fix $m$ to be the dimension of cylindric algebras considered rather than $n$, which we preserve for number of nodes.

From now on $m$, unless otherwise specified, will denote the dimension of a $\CA$.
$m$ will be finite and will be always $\geq 3$.

Notice that the game above is equivalent to $\omega$ rounded atomic games, played on networks, on the algebra, 
where a network is a map $N:{}^m\Delta\to \At\A$  ($\Delta$ is a set of nodes),
satisfying certain conditions (to be recalled below). 
There is a very natural  one to one correspondence between networks on a rainbow algebra 
(and for that matter Monk's algebras based on a class of structures) and coloured graphs in the signature
of this rainbow algebra (Monk algebra). 

When the number of rounds are restricted 
to $k<\omega$, a \ws\ for \pe\ in the network game can be coded in a first order sentence called a Lyndon condition. 
The class of all algebras satisfying all Lyndon conditions is elementary but not finitely axiomatizable 
(Monk-like algebras can prove this, indeed Monk used Lyndon algebras in the his proof that the clas $\RRA$ 
is not finitely axiomatizable; they are bad algebras converging to a good one); 
futhermore, it is is properly contained in the class of completely representable algebras, a delicate distinction,
which Lyndon didn't see. 

This is referred to in the literature of algebraic logic as Lyndon's error, which has caused a lot of confusion among algebraic logicians 
for some time in the past. 
This confusion ended by Hirsch and Hodkinson's result that 
algebras satisfying Lyndon conditions, sure enough are representable, but they {\it may not} be {\it completely representable.}

There are rainbow constructions where \pe\ can win all finite rounded game, but  \pa\ wins the $\omega$ rounded atomic game.
The algebra thereby constructed satisfies the Lyndon conditions
but is not completely representable, though it is elementary equivalent to one that is. 

Let ${\sf LCA_n}$ be the class of 
algebras satisfying the Lyndon conditions, they are necessarily atomic, because they are elementary equivalent to algebras in ${\sf CRA_n}$ 
the class of completely representable algebras, which are atomic, and atomicity is a first order property.
We  have the following strict inclusions lifting them up from atom structures \cite{HHbook2}:
$${\sf CRA}_n\subset {\sf LCA_n}\subset {\sf SRCA_n}\subseteq {\sf WCRA}_n$$  
The second  and fourth classes are elementary but not finitely axiomatizable, bad Monk  algebras converging to a good one,
can witness this, while ${\sf SRCA_n}$ is not closed under both ultraroots and ultraproducts, 
good Monks algebras converging to a bad one witnesses 
this. Rainbow algebras witness that ${\sf CRA_n}$ is not elementary.

Later, we will discuss such inclusions from the perspective of neat embeddings.

Now we start implementing our cylindric rainbow constructions; 
our first result is proving that the class of cylindric algebras of dimension $m$ having an $n>m, n>4$ 
square complete representation is not elementary.

Let $p<\omega$, and $I$ a linearly irreflexive ordered set, viewed as model to a signature containg a binary relation $<$.
$M[p,I]$ is the disjoint union of $I$ and the complete graph $K_p$ with $p$ nodes.
$<$ is interpreted in this structure as follows $<^{I}\cup <{}^{K_p}\cup I\times K_p)\cup (K_p\times I)$
where the order on $K_p$ is the edge relation.

We now consider the rainbow construction where the  the greens are indexed by elements in $A=M[n-4, \Z]$ and the reds indexed
by pairs from $B=M[n-4,\N]$, and everything else is the same.

The game $\EF_r^p[A,B]$, for any $A, B$, is the pebble games between two relational structures  $A$ and $B$,
as defined in \cite{HHbook}. This game is a pebble game,
with $p$ pebbles $\leq \omega$ and $r$ rounds.
Roughly, pairs of pebbles are outside the game and \pe\ and \pa\ choose pebbles consecutively from the same pair,
\pa\ on $\A$ and \pe\ on $\B$, \pa\ wins if the resulting relation from the played pairs is not a partial homomorphisms, else
the game goes for another round.

These games lift to rainbow relation algebras, the pebbles that $A$ plays appears as indices of the greens, while those played by \pe\ appears
as double indices of the reds. Here we show that in certain cases,
such game can also lift to rainbow cylindric algebras whose atoms are coloured graphs,
not the atoms with indices from the two structures.
$G^n_k$ is the usual atomic game with $n$ nodes and $k$ rounds, where \pe\ is required to respond to cylindrifier moves, played on coloured graphs,
In all case considered a win for either player using $p$ pebbles in $k$ rounds, transfers to a \ws\ in $\CA_{A,B}$ 
using $p+2$ nodes and in $k+1$ rounds.

\begin{theorem}\label{she} \pe\ can win $G^n_k$ for every finite $k$ on the $m$ dimensional $\CA_{A,B}$
\end{theorem}
\begin{proof}
She plays her private pebble game $\EF_{k-1}^{n-2}(A,B)$. She can win this game, according
to the following strategy. \pa\ picks up a spare pebble pair and place the first pebble of it on $a\in A$.
By the rules of the game, $a$ is not currently occupied by a pebble. \pe\ has to choose which element of $B$ to put the
pebble on.  \pe\ chooses an unoccupied element in $n-4$, if possible. If they are all already occupied,
she chooses $b$ to be an arbitrary element
$x\in \N$. Because there are only $n-3$ pebble pairs, \pe\ can always implement this strategy and win.
We lift her \ws\ of  the same game but now played on coloured
graphs, the atoms of $\CA_{A,B}$. Denote the class of coloured graphs by $\GG$.
Let $\Gamma$ be a coloured graph built at some stage $t<k$.

We assume inductively that \pe\ has never chosen $\g$ or $\w$  and if $F$ is a face in $\Gamma$, $|F|=m-1$
$\alpha, \beta\in F$, are apexes of two cones inducing the same order on $F$,
then the red clique obtained by considering the reds labelling each two distinct apexes of all such cones, played so far,
has at least $2$ elements.
Let \pa\ choose the graphs $\Phi$ with distinct nodes $F\cup \{\delta\}$ where $F\subseteq \Gamma$ has size
$m-1$. Recall that here $m$ is the dimension of $\CA_{A,B}$. As before we may view \pa\ s 
move as building a coloured graph $\Gamma^*$ extending $\Gamma$
whose nodes are those of $\Gamma$ together with a new node $\delta$ and whose edges are edges of $\Gamma$ together with edges
from $\delta$ to every node of $F$. 
Now \pe\ must extend $\Gamma^*$ to a complete graph on the same nodes and
complete the colouring giving  a graph $\Gamma^+$ in $G$.
In particular, she has to define $\Gamma^+(\beta, \delta)$ for all nodes
$\beta\in \Gamma\sim F$.

\begin{enumarab}

\item  if $\beta$ and $\delta$ are both apexes of two cones on $F$; this is the hardest case. 
Assume that the tint of the cone determined by $\beta$ is $a\in A$, and the two cones
induce the same linear ordering on $F$. Recall that we have $\beta\notin F$, but it is in $\Gamma$, while $\delta$ is not in $\Gamma$,
and that $|F|=m-1$.
Now \pe\ has no choice but to pick a  red colour, she does this as follows:
Let $$R_{\Gamma}(F)=\{x\in \Gamma: \Gamma(x,\delta)=\g_0^i,\text { for some $i\in A$},  F\cup \{x\}$$
$$\text { is an $i$ cone with base $F$ and appex $x$}\}.$$
This is a red clique, it is basically the appexes of cones in $\Gamma$, inducing the same order on $F$.
Since $|\Gamma|<n$ (we have only $n$ rounds), we have $|R_{\Gamma}(F)|< n-2$, hence there are  fewer than $n-2$ pairs of pebbles in play.
\pe\ picks up a spare pebble  pair, so this increases the number of pebbles used (that is nodes)
by $2$, and playing
the role of \pa\ places one of the pebbles in the pair on $a$.
She uses her \ws\ to respond by placing the other
one on $b\in B$.
She then labels the edge between $\beta$ and $\delta$ with $\r_{\mu(\beta), b}$.

\item Other wise, this is not the case, so for some $i<n-1$ there is no $f\in F$ such
that $\Gamma^*(\beta f), \Gamma (f,\delta)$ are both coloured $\g_i$ or if $i=0$, they are coloured
$\g_0^l$ and $\g_0^{l'}$ for some $l$ and $l'$.
\end{enumarab}
In the second case \pe\ uses the normal strategy in rainbow constructions.
She chooses $\w_i$, for $\Gamma{(\beta,\delta)}$, for definiteness let it be the least such $i$.

Now we turn to coluring of $n-1$ tuples.
For each tuple $\bar{a}=a_0,\ldots a_{n-2}\in \Gamma^{n-1}$ with no edge
$(a_i, a_j)$ coloured green, then  \pe\ colours $\bar{a}$ by $\y_S$, where
$$S=\{i\in A: \text { there is an $i$ cone in $\Gamma$ with base $\bar{a}$}\}.$$
We need to check that such labeling works.

Let us check that $(n-1)$ tuples are labeled correctly, by yellow colours.
Let $D$ be  set of $n$ nodes, and suppose that $N\upharpoonright D$
is an $i$ cone with apex
$\delta$ and base $\{d_0,\ldots d_{n-2}\}$, and that the tuple $(d_0,\ldots d_{n-2})$ is labelled $\y_S$ in $N$.
We need to show that $i\in S$. If $D\subseteq N$, then inductively the graph
$N$ constructed so far is in $\bold J$, and therefore
$i\in S$. If $D\subseteq \Phi$ then as \pa chose $\Phi$ in $\bold J$ we get also $i\in S$. If neither holds, then $D$ contains $\alpha$
and also some
$\beta\in N\sim \Phi$. \pe\ chose the colour $N^+(\alpha,\beta)$ and her strategy ensures her that it is green.
Hence neither $\alpha$ or $\beta$ can be the apex of the cone $N^+\upharpoonright D$,
so they must both lie in the base $\bar{d}$.
This implies that
$\bar{d}$ is not yet labelled in $N^*$  ($N^*$'s underlying set is $N$ with the new node, and $N^+$ is the complete labelled graph with nodes $N^*$), 
so \pe\ has applied her strategy to choose the colour $\y_S$ to label $\bar{d}$ in $N^+$.
But this strategy will have chosen $S$ containing $i$ since $N^*\upharpoonright D$ is already a cone
in $N^*$.  Also \pe\ never chooses a green edge, so all green edges of $N^+$ lie in $N^*$.

That leaves one (hard) case, where there are two nodes $\beta, \beta',
\in N$, \pe\ colours both $(\beta, \alpha)$ and $(\beta',
\alpha)$ red, and the old edge $(\beta, \beta')$ has already been
coloured red (earlier in the game).
If $(\beta, \beta')$ was coloured by \pe\ , that is \pe\ is their owner, then there is no problem.
We show that this is what hapened.

So suppose, for a contradiction, that $(\beta, \beta')$ was coloured by
\pe\. This is esentially the argument in \cite{HHbook} proving that \pe\ is indeed the owner. 
Since \pe\ chose red colours for $(\alpha, \beta)$
and $(\alpha, \beta')$, it must be the case that there are cones in
$N^*$ with apexes $\alpha, \beta, \beta'$ and the same base,
$F$, each inducing the same linear ordering $\bar{f} = (f_0,\ldots,
f_{n-2})$, say, on $F$. Of course, the tints of these cones may all
be different. Clearly, no edge in $F$ is labelled green, as no cone base can contain
green edges. It follows that $\bar{f}$  must be labeled by some
yellow colour, $\y_S$, say. Since $\Phi\in \bold J$, it obeys its definition, so
the tint $i$ (say) of the cone from $\alpha$ to $\bar{f}$ lies in
$S$. Suppose that $\lambda$ was the last node of $ F \cup \{ \beta,
\beta' \}$ to be created,as the game proceeded. As $ |F \cup \{
\beta, \beta' \}| = n + 1$, we see that \pa\ must have chosen
the colour of at least one edge in this : say, $( \lambda, \mu )$.
Now all edges from $\beta$ into $F$ are green,  so \pe\ is the owner of them
as well as of  $(\beta, \beta')$.

The same holds for edges from $\beta'$ to $F$. Hence $\lambda, \mu \in F$.
We can now see that it was \pe\ who chose the colour $\y_S$ of
$\bar{f}$. For $\y_S$ was chosen in the round when $F$'s last node,
i.e., $\lambda$ was created. It could only have been chosen by
\pa\ if he also picked the colour of every edge in $F$
involving $\lambda$. This is not so, as the edge $(\lambda, \mu)$
was coloured by \pe\, and lies in $F$.
As $i \in S$, it follows from the definition of \pa\'s strategy
that at the time when $\lambda$ was added, there was already an
$i$-cone with base $\bar{f}$, and apex $N$ say.
We claim that $ F \cup \{ \alpha \}$ and $ F \cup \{ N \}$ are
isomorphic over $F$. For this, note that the only $(n - 1)$-tuples of
either $ F \cup \{ \alpha \}$ or $ F \cup \{ N \}$ with a
yellow colour are in $F$ ( since all others involve a green edge
). But this means that \pe\ could have taken $\alpha = N$ in
the current round, and not extended the graph. This is contrary to
our original assumption, and completes the proof.

\end{proof}

\begin{theorem}\label{he}\pa\ can win $G^n_{\omega}$ on $\CA_{A,B}$
\end{theorem}
\begin{proof}
In her private game, \pa\ always places the pebbles on distinct elements of $\Z$.
She uses rounds $0,\ldots n-3$, to cover $n-4$ and first two elements of $\Z$.
Because at least two out of three distinct colours are related by $<$, \pe\ must respond by pebbling
$n-4 \cup \{e,e'\}$ for some  $e,e'\in \N$.
Assuming that \pa\ has not won, then he has at least arranged that two elements of $\Z$ are pebbled, the corresponding pebbles in $B$ being in $\N$.
Then \pa\ can force \pe\ to play a two pebble game of length $\omega$ on $\Z$, $\N$
which he can win, bombarding him with cones with green tints, in the graph game.

Assume that $p\geq 3$ (for $p\leq 2$, the game is degenarate), and inductively that these are nodes  $n_0,\ldots n_{q-1}$,
$q<\omega$, $2\leq q\leq p$, added for cones  with tints $j$, inducing the same order on one face,
and  $a_j\in A$, $j<p$, $a_j$ pairwise distinct  and the indices on the nodes $(n_j, n_k)$ must be red, $j,k<q$, so that $\{n_j: j<q\}$ forms a red clique.
Each node $n_j$ has an index $\beta(n_j)\in B$. As part of the inductive hypothesis,
suppose that at the start of round $t-1$ of $\EF_{\omega}^{n-2}(A,B)$, \pe\ has not lost yet, so that \pa\ is still using his \ws\, and
the situation corresponds to the situation in round $t$ of $G^n_{\omega}$. That is there is
a pair of pebbles on $(a_j, \beta(n_j))$ for each $j<q.$

We can assume that \pa\ only removes a single pair and only when he has to.  If the number $q$ of pebbles is already $p$
then \pa\ removes a pair of pebbles say on $(a_j \beta(n_j)$ for some
$j<q$. In this case there must be at least two distinct pebbles so the $\EF_{\omega}^{n-2}(A,B)$ game goes on for at least
another move.

Now we know that  \pa\ s has a winning strategy for
$\EF_r^p(A,B)$
If \pa\ s strategy in this game tells him to place a pebble a on $A$, then in the graph game he plays the  cone with base $F$, $|F|=m-1$,
and tint $a$.
In the graph game, in the next round,  he picks the same base and  the cone with tint $a.$
This forces \pe\ to add a new node $n$ to the graph. Then $n$ must be part of a red clique.
So $n$ has an index $\beta(n)\in B$. In \pa\ private game he lets \pe\ place her
corresponding pebble on $\beta(n)$ Because \pa has a \ws\ in his private game,
eventually he will place a pebble $a\in A$, but there is nowhere in $B$ for \pe\ to place the other pebble.
But this means that $\{a_j, \beta(n_j)), (a_j' ,\beta(n_j')\}$ is not a partial homomorphism.
hence $(n_j, n_j', c)$ is not consistent, and \pa\ has won.

This is the usual strategy for \pa\ to win, using her greens successively to  create cones with the same  base forcing \pe\
to play a  red clique, eventually running out of reds  one way or  another.
\pa\ has a \ws\ in the $\omega$ rounded game
by bombarding \pe\ with cones on the same base and different green tints.

\end{proof}
The atomic $\omega$ rounded games for both relation and cylindric algebras test complete representability.
For relation algebras when we restrict the nodes
to $n\geq 4$, then \pe\ has a \ws\ over an atomic algebra if and only if it has an $n$ dimensional relational bases;
the class of all such algebras turns out
to be a variety with standard notation. Maddux shows that $\RA_n\neq \RA_{n+1}$,
and that such varieties constitute a strict approximation to $\RRA$ in the sense that
$\bigcap_n \RA_n=\RRA$.
Varying the parameters, namely, the number of rounds and nodes
one can obtain more sophisticated delicate results
like $\RA_{n}$ is not finitely axiomatizable over $\RA_{n+1}$, a result of Hirsch and Hodkinson.
If we restrict the nodes to be finite for cylindric algebras, atomic algebras
for which \pe\ has a \ws\ here do not give a variety, instead they give algebras that have  $n$ square complete
representations and this class is not even elementary,
it is rather an approximation of the class of completely representable algebras.

\begin{definition}
\begin{enumarab}
\item  Let $M$ be a relativized representation of a $\CA_m$.
A clique in $M$ is a subset $C$ of $M$ such that $M\models 1(\bar{s})$ for all $\bar{s}\in {}^mC$.
For $n>m$, let $C^{n}(M)=\{\bar{a}\in {}^nM: \text { $\rng(\bar{a})$ is a clique in M}\}.$
\item Let $\A\in \CA_m$, and $M$ be a relativized representation of $\A$. $\M$ is said to be $n$ square, $n>m$, 
if whenever $\bar{s}\in C^n(M)$, $a\in A$, and $M\models {\sf c}_ia(\bar{s}$, 
then there is a $t\in C^n(M)$ with $\bar{t}\equiv _i \bar{s}$, 
and $M\models a(\bar{t})$.
\end{enumarab}
\end{definition}
We note that the clique relativiuzed semantics is related to the clique guarded fragments of first order logic \cite{HHbook}. 
$M$ is a complete $n$ square relativized semantics of an atomic $\A$, if whenever $\bar{a}\in C^{m}M$, there is an atom in $\A$ such
that $M\models a(\bar{s})$. (We note that atomicity here is redundant).

\begin{theorem} Let $\A\in \CA_m$ be atomic. Then the following are equivalent:
\begin{enumarab}
\item   $\A\in S_c\Nr_m\CA_n$
\item $\A$ has an complete relativized $n$ square representation.
\item \pe\ has a \ws\ in $G_m^n$.
\end{enumarab}
\end{theorem}
\begin{proof}
Assume that $\A\subseteq \Nr_m\C$, let $\Lambda=\bigcup_{k<n-1}\At\Nr_k\C$, and let $\lambda\in \Lambda$.
For each atom $x$ of $\C$, define $N_x$, easily checked to be an $m$ dimensional   $\Lambda $hypernetwork, as follows.
Let $\bar{a}\in {<}^{n}n$. Then if $|a|=m$ $N_x(a)$ is the unique atom $r\in \At\A$ such that $x\leq {\sf s}_{\bar{a}}r$.
Here substitutions are defined as above.
If $n\neq |\bar{a}| <m-1$, $N_x(\bar{a})$ the unique atom $r\in \Nr_{|a|}\C$ such that $x\leq s_{\bar{a}}.$
$\Nr_{|a|}\C$ is easily checked to be atomic, so this is well defined.
Otherwise, put  $N_x(a)=\lambda$.
Then $N_x$ as an $m$ dimensional $\Lambda$ network, for each such chosen $x$ and $\{N_x: x\in \At\C\}$
is an $m$ dimensional $\Lambda$ hyperbasis.
Then viewing those as a saturated set of mosaics, one can can construct complete
$n$ square representation of $\A$. Alternatively, one can use a standard step by step argument.
Conversely, assume that $\A$ has an $n$ square complete representation $M$. For $\phi\in L(A)_{\omega,\infty}^n$,
let $\phi^M=\{\bar{a}\in C^n(M): M\models \phi(\bar{a})\}$, and let $\D$ be the algebra with univerese $\{\phi^M: \phi\in L\}$ with usual
Boolean and cylindrifiers. Then this is a $\CA_n$, here semantics is defined as expected in the clique guarded fragment of first order logic.
Define $\D_0$ be the algebra consisting of those $\phi^M$ where $\phi$ comes from $L$.
Then $\D_0$ is also a $\CA_n$ and $\A$ embeds into the $m$ neat reduct of both.
If $\M$ is complete, then the embedding is also complete.
The equivalence of (2) and (3) is proved in a, by now, fairly standard way concerning such equivalences.
Basically a complete relativized representation guides \pe\ to a \ws\, and conversely if \pe\
has a \ws\ in $G^m_{n}$, then for every atom $a\in \A$, consider a play of the game in which \pe\ plays networks
with fewer that $n$ nodes, all hyperedges and all legitimate atoms at some stage ogf the game eventually.
Let the limit of this play be $N_a$, then  $h(b)=\{x:\exists a\in \At\A: x\in {}^nN_a, N_a(x)\leq b\}$ is an $m$ square complete
representation as desired.
\end{proof}

\begin{corollary} For $n\geq 5$, the class ${\sf CRA_{m,n}}$ is not elementary
\end{corollary}

\begin{proof} Since \pe\ has a \ws\ for all finite rounded games (with $n$ nodes), she has a \ws\ on the ultrapower, which has an $n$
complete representation.
But \pa\ wins the $\omega$ rounded game, also with $n$ nodes,
hence $\A$ does not have an $n$ complete relativized representation, but is elementary equivalent to one that does.
\end{proof}

Our next $\CA_m$ is $\A_r^n$, the rainbow cylindric algebra based on $A[n,r]=M[-3, 2^{r-1}],$
and $B=M[n-3, 2^{r-1}-1].$  $n$ is the number of rounds so we have $n-2$ pebbles.

\begin{lemma}
\begin{enumarab}
\item \pe\ has a \ws\ in the game $G_{\omega}^n(\A_r^n)$
\item \pe\ has a \ws\ in $G_r^{\omega}(\A_r^n)$.
\end{enumarab}
\end{lemma}
\begin{proof} Like before, theorem \ref{she}, where \pe\ uses his \ws\ in the private game $EF(A,B)$, choosing the
red label between two nodes $\delta, \beta$ being apexes of two cones, the former with tint $a$
and inducing the same order on $F$ as $\r_{\beta}$  where $\beta$ is the index of
the red clique defined above.
\end{proof}

\begin{theorem} \pa\ has a \ws\ in $G_{\omega}^{n+1}(\A_r^n)$
\end{theorem}
\begin{proof}
Also like \ref{he}. He uses her private game which is $\EF_{\omega}^{n-1}(A, B)$.
Then he picks $\w$ and she plays sucessively cones,
with green tints, forcing \pe\ to play a red clique on the base.

In her private game, \pa\ always place the pebbles on distinct elements of $A$.
he uses rounds $0,\ldots n-2$, to cover $K(A)$ and the elements
$l-1, l-2\in I(A)$.  Then \pa\ can force \pe\ to play a two pebble game of length $\omega$ on $I(A)$ and $I(B)$ which he can win because $I(A)$ is
longer than $I(B)$, bombarding him with cones having the same base, namely the induced face by \pa\ move,
and different tints, in the graph game.
\end{proof}
We know that $\A_r^n$ for any $r$ separates $\CRA_{n,m}$ from $\CRA_{n+1,m}$. But it is possible that an ultrapower or an ultraproduct or both applied to
$\A_r^n$ is in $UpUr\CRA_{n,m}$, and indeed we have \pe\ can win  $G_{\omega}^{n+1}[\prod_r \A_r^n/D]$. The ultraproduct is an atomic algebra,
and it belongs to $\CRA_{n+1,m}$.

\section{Classes of subneat reducts that are not elementary, and not closed under completions}

Here we change our notation to the more conventional one, namely, cylindric algebras of dimension
$n$, will be denoted by $\CA_n$.

Let $\A$ be the complex algebra over $\CA_{\Z, \N}$. Then $\A$ is representable because \pe\ can win the finite rounded games \cite{hh}.
Note that it is not completely representable because \pa\ can win the $\omega$ rounded game.
Now consider the following game played on networks, and then translated to coloured graphs:
We need some preliminaries.

\begin{definition}\label{def:string}
Let $n$ be an ordinal. An $s$ word is a finite string of substitutions $({\sf s}_i^j)$,
a $c$ word is a finite string of cylindrifications $({\sf c}_k)$.
An $sc$ word is a finite string of substitutions and cylindrifications
Any $sc$ word $w$ induces a partial map $\hat{w}:n\to n$
by
\begin{itemize}

\item $\hat{\epsilon}=Id$

\item $\widehat{w_j^i}=\hat{w}\circ [i|j]$

\item $\widehat{w{\sf c}_i}= \hat{w}\upharpoonright(n\sim \{i\}$

\end{itemize}
\end{definition}

If $\bar a\in {}^{<n-1}n$, we write ${\sf s}_{\bar a}$, or more frequently
${\sf s}_{a_0\ldots a_{k-1}}$, where $k=|\bar a|$,
for an an arbitrary chosen $sc$ word $w$
such that $\hat{w}=\bar a.$
$w$  exists and does not
depend on $w$ by \cite[definition~5.23 ~lemma 13.29]{HHbook}.
We can, and will assume \cite[Lemma 13.29]{HHbook}
that $w=s{\sf c}_{n-1}{\sf c}_n.$
[In the notation of \cite[definition~5.23,~lemma~13.29]{HHbook},
$\widehat{s_{ijk}}$ for example is the function $n\to n$ taking $0$ to $i,$
$1$ to $j$ and $2$ to $k$, and fixing all $l\in n\setminus\set{i, j,k}$.]
Let $\delta$ be a map. Then $\delta[i\to d]$ is defined as follows. $\delta[i\to d](x)=\delta(x)$
if $x\neq i$ and $\delta[i\to d](i)=d$. We write $\delta_i^j$ for $\delta[i\to \delta_j]$.

We recall the definition of network:
\begin{definition}
From now on let $2\leq n<\omega.$ Let $\C$ be an atomic $\CA_{n}$.
An \emph{atomic  network} over $\C$ is a map
$$N: {}^{n}\Delta\to \At\cal C$$
such that the following hold for each $i,j<n$, $\delta\in {}^{n}\Delta$
and $d\in \Delta$:
\begin{itemize}
\item $N(\delta^i_j)\leq {\sf d}_{ij}$
\item $N(\delta[i\to d])\leq {\sf c}_iN(\delta)$
\end{itemize}
\end{definition}
Note than $N$ can be viewed as a hypergraph with set of nodes $\Delta$ and
each hyperedge in ${}^n\Delta$ is labelled with an atom from $\C$.
We call such hyperedges atomic hyperedges.

\begin{definition}\label{def:hat}
For $m\geq 5$ and $\C\in\CA_m$, if $\A\subseteq\Nr_n(\C)$ is an
atomic cylindric algebra and $N$ is an $\A$-network then we define
$\widehat N\in\C$ by
\[\widehat N =
 \prod_{i_0,\ldots i_{n-1}\in\nodes(N)}{\sf s}_{i_0, \ldots i_{n-1}}N(i_0\ldots i_{n-1})\]
$\widehat N\in\C$ depends
implicitly on $\C$.
\end{definition}
We write $\A\subseteq_c \B$ if $\A\in S_c\{\B\}$.
\begin{lemma}\label{lem:atoms2}
Let $n<m$ and let $\A$ be an atomic $\CA_n$,
$\A\subseteq_c\Nr_n\C$
for some $\C\in\CA_m$.  For all $x\in\C\setminus\set0$ and all $i_0, \ldots i_{n-1} < m$ there is $a\in\At(\A)$ such that
${\sf s}_{i_0\ldots i_{n-1}}a\;.\; x\neq 0$.
\end{lemma}
\begin{proof}
We can assume, see definition  \ref{def:string},
that ${\sf s}_{i_0,\ldots i_{n-1}}$ consists only of substitutions, since ${\sf c}_{m}\ldots {\sf c}_{m-1}\ldots
{\sf c}_nx=x$
for every $x\in \A$.We have ${\sf s}^i_j$ is a
completely additive operator (any $i, j$), hence ${\sf s}_{i_0,\ldots i_{\mu-1}}$
is too  (see definition~\ref{def:string}).
So $\sum\set{{\sf s}_{i_0\ldots i_{n-1}}a:a\in\At(\A)}={\sf s}_{i_0\ldots i_{n-1}}
\sum\At(\A)={\sf s}_{i_0\ldots i_{n-1}}1=1$,
for any $i_0,\ldots i_{n-1}<n$.  Let $x\in\C\setminus\set0$.  It is impossible
that ${\sf s}_{i_0\ldots i_{n-1}}\;.\;x=0$ for all $a\in\At(\A)$ because this would
imply that $1-x$ was an upper bound for $\set{{\sf s}_{i_0\ldots i_{n-1}}a:
a\in\At(\A)}$, contradicting $\sum\set{{\sf s}_{i_0\ldots i_{n-1}}a :a\in\At(\A)}=1$.
\end{proof}
We define a game on networks, which  has $\omega$  rounds and $m$ pebbles.
If \pa\ wins this game played on networks of an atomic $\A$, this means that
$\A\notin S_c\Nr_n\CA_m$.

\begin{definition}
Let $m\leq \omega$. This is a typical $m$ pebble game.
In a play of $F^m(\alpha)$ the two players construct a sequence of
networks $N_0, N_1,\ldots$ where $\nodes(N_i)$ is a finite subset of
$m=\set{j:j<m}$, for each $i$.  In the initial round of this game \pa\
picks any atom $a\in\alpha$ and \pe\ must play a finite network $N_0$ with
$\nodes(N_0)\subseteq  m$,
such that $N_0(\bar{d}) = a$
for some $\bar{d}\in{}^{n}\nodes(N_0)$.
In a subsequent round of a play of $F^m(\alpha)$ \pa\ can pick a
previously played network $N$ an index $\l<n$, a ``face"
$F=\langle f_0,\ldots f_{n-2} \rangle \in{}^{n-2}\nodes(N),\; k\in
m\setminus\set{f_0,\ldots f_{n-2}}$, and an atom $b\in\alpha$ such that
$b\leq {\sf c}_lN(f_0,\ldots f_i, x,\ldots f_{n-2}).$
(the choice of $x$ here is arbitrary,
as the second part of the definition of an atomic network together with the fact
that $\cyl i(\cyl i x)=\cyl ix$ ensures that the right hand side does not depend on $x$).
This move is called a \emph{cylindrifier move} and is denoted
$(N, \langle f_0, \ldots f_{\mu-2}\rangle, k, b, l)$ or simply $(N, F,k, b, l)$.
In order to make a legal response, \pe\ must play a
network $M\supseteq N$ such that
$M(f_0,\ldots f_{i-1}, k, f_i,\ldots f_{n-2}))=b$
and $\nodes(M)=\nodes(N)\cup\set k$.

\pe\ wins $F^m(\alpha)$ if she responds with a legal move in each of the
$\omega$ rounds.  If she fails to make a legal response in any
round then \pa\ wins. The more pebbles we have, the easier it is for \pa\ to win.
\end{definition}
This game is is like the usual $\omega$ rounded
atomic game played on networks of cylindric algebras $G_{\omega}$
except that the number of nodes used are limited and \pa\ can re-use nodes.
If we allow only $m$ nodes in the cylindric algebra game without allowing \pa\ to reuse nodes,
then the resulting game characterizes those cylindric algebras that have an $n$ square
relativized representation meaning that a win for \pe\ using $n$ pebbles
imply that the algebra has an $n$ square representation and the converse holds
as well. We will return to such issues later.

\begin{theorem}\label{thm:n}
Let $n<m$, and let $\A$ be an atomic $\CA_m$
If $\A\in{\bf S_c}\Nr_{n}\CA_m, $
then \pe\ has a \ws\ in $F^m(\At\A)$. In particular, if $\A$ is countable and completely representable, then \pe\ has a \ws in $F^{\omega}(\At\A)$
\end{theorem}
\begin{proof}
For the first part, if $\A\subseteq\Nr_n\C$ for some $\C\in\CA_m$ then \pe\ always
plays hypernetworks $N$ with $\nodes(N)\subseteq n$ such that
$\widehat N\neq 0$. In more detail, in the initial round , let $\forall$ play $a\in \At \cal A$.
$\exists$ play a network $N$ with $N(0, \ldots n-1)=a$. Then $\widehat N=a\neq 0$.
At a later stage suppose $\forall$ plays the cylindrifier move
$(N, \langle f_0, \ldots f_{\mu-2}\rangle, k, b, l)$
by picking a
previously played hypernetwork $N$ and $f_i\in \nodes(N), \;l<\mu,  k\notin \{f_i: i<n-2\}$,
and $b\leq {\sf c}_lN(f_0,\ldots  f_{i-1}, x, f_{n-2})$.
Let $\bar a=\langle f_0\ldots f_{l-1}, k\ldots f_{n-2}\rangle.$
Then ${\sf c}_k\widehat N\cdot {\sf s}_{\bar a}b\neq 0$.
Then there is a network  $M$ such that
$\widehat{M}.\widehat{{\sf c}_kN}\cdot {\sf s}_{\bar a}b\neq 0$. Hence
$M(f_0,\dots  k, f_{n-2})=b.$

For the second part, we have from the first part, that $\A\in S_c\Nr_n\CA_{\omega}$, the result now follows.
\end{proof}

The main strategy for \pa\ s wins in rainbow games is that he uses his greens to force a red clique that \pe\ cannot cope with.
That is he uses his green atoms, namely cones, forcing  \pe\ to use  red atoms, until she is forced an inconsistency.

\begin{theorem} \pa\ has a winning strategy in $F^{n+2}(\At\CA_{\Z,\N})$
\end{theorem}
\begin{proof}
This is the usual strategy for \pa\ to win, using her greens successively to  create cones with the same  base forcing \pe\
to play a  red clique, eventually running out of reds  one way or  another.
\pa\ has a \ws\ in the $\omega$ rounded game $F^{n+2}$ one with $n+2$ nodes,
by bombarding \pe\ with cones on the same base and different green tints,
forcing a decreasing sequence in $N$.

In his zeroth move, $\forall$ plays a graph $\Gamma \in \bold J$ with
nodes $0, 1,\ldots, n-1$ and such that $\Gamma(i, j) = \w (i < j <
n-1), \Gamma(i, n-1) = \g_i ( i = 1,\ldots, n), \Gamma(0, n-1) =
\g^0_0$, and $ \Gamma(0, 1,\ldots, n-2) = \y_\omega $. This is a $0$-cone
with base $\{0,\ldots , n-2\}$. In the following moves, $\forall$
repeatedly chooses the face $(0, 1,\ldots, n-2)$ and demands a node (possibly used before)
$\alpha$ with $\Phi(i,\alpha) = \g_i (i = 1,\ldots,  n-2)$ and $\Phi(0, \alpha) = \g^\alpha_0$,
in the graph notation -- i.e., an $\alpha$-cone on the same base.
$\exists$, among other things, has to colour all the edges
connecting nodes. The idea is that by the rules of the game
the only permissible colours would be red. Using this, $\forall$ can force a
win eventually for else we are led to a a decreasing sequence in $\N$.

In more detail,
In the initial round $\forall$ plays a graph $\Gamma$ with nodes $0,1,\ldots n-1$ such that $\Gamma(i,j)=\w$ for $i<j<n-1$
and $\Gamma(i,n-1)=\g_i$
$(i=1, \ldots n-2)$, $\Gamma(0,n-1)=\g_0^0$ and $\Gamma(0,1\ldots n-2)=\y_{N}$.
$\exists$ must play a graph with $\Gamma_1(0,\ldots n-1)=\g_0$.
In the following move $\forall$ chooses the face $(0,\ldots n-2)$ and demands a node $n$
with $\Gamma_2(i,n)=\g_i$ and $\Gamma_2(0,n)=\g_0^{-1}.$
$\exists$ must choose a label for the edge $(n,n-1)$ of $\Gamma_2$. It must be a red atom $r_{mn}$. Since $-1<0$ we have $m<n$.
In the next move $\forall$ plays the face $(0, \ldots n-2)$ and demands a node $n+1$ such that  $\Gamma_3(i,n+1)=\g_i^{-2}$.
Then $\Gamma_3(n+1,n)$ $\Gamma_3(n+1,n-1)$ both being red, the indices must match.
$\Gamma_3(n+1,n)=r_{ln}$ and $\Gamma_3(n+1, n-1)=r_{lm}$ with $l<m$.
In the next round $\forall$ plays $(0,1\ldots n-2)$ and reuses the node $n-2$ such that $\Gamma_4(0,n-2)=\g_0^{-3}$.
This time we have $\Gamma_4(n,n-1)=\r_{jl}$ for some $j<l\in N$.
Continuing in this manner leads to a decreasing sequence in $\N$.
\end{proof}

(Notice that here \pa\  needed at least $n+2$ pebbles.
The number of pebbles, $k>n$ say, necessary for \pa\ to win the game,
excludes {\it complete} neat embeddability of $\A$ in
an algebra with $k$ dimensions.)

\begin{corollary} The algebra $\A$ (definition above)
is not in ${\bf S_c}\Nr_{n}\CA_{n+2}$.
\end{corollary}

\begin{corollary} The algebra $\A$ is not completely representable
\end{corollary}

\begin{theorem} The omitting types theorem fails for even $n+2$ square representations.
\end{theorem}
\begin{demo}{proof} Let $\A$ be an atomic countable representable algebra that is not in $S_c\Nr_n\CA_{n+2}$.
Let $\Gamma$ be the set of co-atoms,
then it does not have a $k$ square complete representation.
\end{demo}

\begin{theorem} Let $\A=\CA_{n+2, n+1}$. Then  \pa\ has a \ws\ in $n+4$ rounds in the usual atomic rounded atomic  game on graphs.
\end{theorem}
\begin{proof}
She plays like she did before, playing (green) cones with yellow base
forcing \pe\ to run out of  reds. Viewed differently, and indeed
more simply,  \pa\ has a \ws\ in $EF_r^p({\sf G},{\sf R})$, for any $p,r\geq n+2$.
In each round $0,1\ldots n+2$ he places a  new pebble  on  element of $n+2$.
The edges relation in $n+1$ is irreflexive so to avoid losing
\pe\ must respond by placing the other  pebble of the pair on an unused element of $n+1$.
After $n+1$ rounds there will be no such element,
and she loss in the next round. Hence \pa\ can win the graph game using $n+4$ pebbles.
\end{proof}

Now we split ever red to infinitely many copies obtaining the new class $\bold J$ consisting of coloured graphs
with the following properties.
\begin{definition}
\begin{enumarab}

\item $M$ is a complete graph.

\item $M$ contains no triangles (called forbidden triples)
of the following types:

\vspace{-.2in}
\begin{eqnarray}
&&\nonumber\\
(\g, \g^{'}, \g^{*}), (\g_i, \g_{i}, \w), \\
\label{forb: black}(\y,\y,\y), \\
\label{forb:match}(\r_{ij}^i, \r_{j'k'}^{i'}, \r_{i^*k^*}^{i^*})&&\mbox{unless }i=i^*,\; j=j'\mbox{ and }k'=k^*\\
\label{form}(\r_{ij}^i, \r_{j'k'}^{i'}, \rho)
\label{f} (\r_{ij}^i, \rho, \rho)
\end{eqnarray}

\item The second and third item like before.

\end{enumarab}
\end{definition}
Let $\GG$ denotes the set of all coloured graphs.
The next theorem, due to Hodkinson, is the cornerstone of our result, since it determines the model on which our term algebra will be based.
However, like our very first blow up and blur construction (applied to Monk algebras)
we will not use all assignments, we will have discard the assignments for which an edge is labelled by
$\rho$. This gives a relativized representation, but it is isomorphic to a set algebra that has a square one, so
it will be representable. Furthermore, we will show that it is atomic, and its atoms are precisely the $n$ surjections to coloured graphs, without
edges labelled by $\rho$.

\begin{theorem}
There is a countable coloured  $M\in \GG$ with the following
property:\\
$\bullet$ If $\triangle \subseteq \triangle' \in \GG$, $|\triangle'|
\leq n$, and $\theta : \triangle \rightarrow M$ is an embedding,
then $\theta$ extends to an embedding $\theta' : \triangle'
\rightarrow M$.
\end{theorem}
\begin{proof}\cite{Hodkinson}
\end{proof}

Take, like in our first blow up and blur construction, $W\subseteq {}^nM$, by roughly dropping assignments whose edges ar not labelled
by $rho.$
Formally, $W = \{ \bar{a} \in {}^n M : M \models ( \bigwedge_{i < j < n,
l < n} \neg \rho(x_i, x_j))(\bar{a}) \}.$

The term algebra call it $\A$, has universe $\{\phi^M: \phi\in L^n\}$ where $\phi^M=\{s\in W: M\models \phi[s]\}.$
Here $\phi^M$ denotes the permitted assignments satisfying $\phi$ in $M$.
Its completion is the relativized set algebra $\C$ which has universe the larger $\{\phi^M: \phi\in L^n_{\infty,\omega}\}$,
which is not representable. (All logics are of course taken in rainbow signature).
The isomorphism from  $\Cm\At\A$ to $\C$ is given by $X\mapsto \bigcup X$.

\begin{corollary} The class $S\Nr_n\CA_{n+4}$ is not closed under completions
\end{corollary}

\begin{proof} Let $\A$ be the above rainbow finite algebra. The certainly \pa \ has a \ws\ in $F^{n+4}$, so that $\A\notin S_c\Nr_n\CA_{n+4}$. But
$\A=\A^+$, hence $\A\notin S\Nr_n\CA_{n+4}$, for if it were, then $\A^+$ would be in $S_c\Nr_n\CA_{n+4}$ and this is not the case.
Split each red  $\r_{ij}$ into $\omega$ many copies $r_{ij}^l$, $l\in \omega$, and a add a shade of red $\rho$, then we get
new infinite countable atom structure $\alpha$.
The term algebra on $\alpha$ is representable, but $\Cm\alpha\notin S\Nr_n\CA_{n+4}$ since $\A$ embeds into it via
$\r_{ij}\mapsto \bigvee_{l\in \omega} r_{ij}^l$.
\end{proof}

Viewed differently,  while looking at $M$ as an $n$- homogeneous model for the rainbow signature, we have:

\begin{theorem}$\C$ does not have an $n+4$ square representation.
\end{theorem}

\begin{proof} Assume that $g:\C\to \wp(V)$ is such. Then $V\subseteq {}^nN$ and we can assume that
$g$ is injective because $\C$ is simple. First there are $b_0,\ldots b_{n-1}\in N$ such $\bar{b}\in h(y_{n+2}(x_0,\ldots x_{n-1}))^W$.
This tuple will be the base of finitely many cones, that will be used to force an inconsistent triple of reds.
This is because $y_{n+2}(\bar{x})^W\neq \emptyset$.  For any $t<n+3$, there is a $c_t\in N$, such
that $\bar{b}_t=(b_0,\ldots b_{n-2}, c_t)$ lies in $h(g_0^t(x_0, x_{n-1})^W$ and in $h(g_i(x_i, x_{n-1})^W)$ for each $i$ with
$1\leq i\leq n-2$. The $c_t$'s are the apexes of the cones with base $\y_{n+2}$.

Take the formula
$$\phi_t=y_{n+2}(x_0,\ldots ,x_{n-2})\to \exists x_{n-1}(g_0^t(x_0, x_{n-1}))\land \bigwedge_{1\leq i\leq n-2} g_i(x_i, x_{n-1})),$$
then $\phi_t^{W}=W$. Pick $c_t$ and $\bar{b_t}$ as above, and define for each $s<t<n+3,$ $\bar{c_{st}}$ to be
$(c_s,b_1,\ldots  b_{n-2}, c_t)\in {}^nN.$
Then $\bar{c}_{st}\notin h((x_0, x_{n-1})^W$. Let $\mu$ be the formula
$$x_0=x_{n-1}\lor w_0(x_0, x_{n-1})\lor \bigvee g(x_0, x_{n-1}),$$
the latter formula is a first order formula consisting of the disjunction of  the (finitely many ) greens.
For $j<k<n$, let $R_{jk}$ be the $L_{\infty\omega}^n$-formula $\bigvee_{i<\omega}r_{jk}^i(x_0, x_{n-1})$.
Then
$\bar{c}_{st}\notin h(\mu^W)$, now for each $s<t< n+3$, there are $j<k<n$ with $c_{st}\in h(R_{jk})^W.$
By the pigeon- hole principle, there are $s<t< n+3$ and $j<k<n$
with $\bar{c}_{0s}, \bar{c}_{0t}\in h(R_{jk}^W)$. We have also $\bar{c}_{st}\in h(R_{j',k'}^W)$
for some $j', k'$ then the sequence $(c_0, c_s,\ldots,  b_2 b_{n-2}, c_t)\in h(\chi^W)$
where
$$\chi=(\exists_1R_{jk})(\land (\exists x_{n-1}(x_{n-1}=x_1\land \exists x_1 R_{jk})\land \exists x_0(x_0=x_1)\land \exists x_1R_{j'k})),$$
so $\chi^W\neq \emptyset$. Let $\bar{a}\in \chi ^W$. Then $M\models _W  R_{jk}(a_0,a_{n-1})\land R_{jk}(a_0,a_1)\land R_{j'k'}(a_1, a_{n-1})$.
Hence there are
$i$, $i'$ and $i''<\omega$ such that
$$M\models _Wr_{jk}^{i}(a_0,a_{n-1})\land r_{jk}^{i'}(a_0,a_1)\land r_{j'k'}^{i''}(a_1, a_{n-1}).$$
But this triangle is inconsistent.
Note that this inconsistent red was forced by an $n+4$ red clique
labeling apexes of the same cone, with base $\y_{n+2}$.
\end{proof}

\begin{theorem} Any class $K$, such that $\Nr_n\CA_{\omega}\subseteq K\subseteq S_c\Nr_n\CA_{n+2}$ is elementary
\end{theorem}
\begin{proof} This is the cylindric analogue of the construction in \cite{r}. The cylindric atom structure is based on the rainbow signature
where the cylindric algebra is $\CA_{\Z, \N}$.
Two games on networks are defined.  $F^m$ is defined on networks which translate equivalently to games on
coloured graphs, as above.  The second game $H$ is played on hypernetworks, so we have a pair, a network and hyperedges.
Some are defined to be short and the other are long. Short hyperedes are constantly labelled, this is called a
$\lambda$  neat hypernetwork. The game that  has three moves by \pa\ is played on $\lambda$ neat hypernetworks.
This translates to hypergraphs with  hyperlabels.
The first move is like $F^m$ except that there is no restriction  on the number of nodes, 
so it is the usual atomic game, as that in Hirsch and Hodkinson \cite{hh}.
The second and third moves by \pa\ are amalgamation moves,
because the game captures a two sorted structure, namely, a full neat reduct.

Roughly, the hyperedges get longer and longer, and the amalgamation move says that the algebra based on the atom structure
has a $k$ hyperbasis for every finite $k$,
the short hyperedges pin down the $n$ neat reduct, ensuring that their atoms are no smaller than the big algebra, which
the arbitrarily long hyperedges capture (in $\omega$ extra dimensions) in the limit.
(See the last part of the proof).
\pe\ has a \ws\ in the $k$ rounded games, for every finite $k$, using the standard rainbow strategy in response to the first move by \pa\,
she uses reds for apexes of the same cones,  the suffixes  of the red used is uniquely determined by the red clique as defined above,
that arises from the family of cones having base $F$, the face played in \pa\ s move, and apex $x$, that is
$R_{\Gamma_i}(F)$ where $\Gamma_i$ is the extended graph.
Otherwise, she uses white. Her response to amalgamation moves is similar.

 As we have seen, \pa\ has a \ws\ in the $\omega$ rounded game $F^{n+2}$,
by bombarding \pe\ with cones on the same base and different green tints,
forcing a decreasing sequence in $\N$. This implies that \pa\ has a \ws\ also in the $\omega$ rounded game in $H$

The hardest part for \pe\  is the usual, when \pa\ produces two nodes that are apexes of the same cone
inducing the same order on a face $F$, and \pe\ has to label this edge red.
In this case \pe\ defines the red clique in the graph $\Gamma$ to be extended,  $R_{\Gamma}(F)$ to which $\beta, \delta$ belong.
she plays her private game $EF^{k-2}_{\omega}(\Z, \N)$ by playing the role of \pa\
putting the pebble $a\in A$ where $a$ is the tint
of the cone induced by $\delta$,  she uses her \ws\ in the private game, finding an unpebbled $b$,
then she labels this edge by $\r_{\mu(\beta), b}$.

Winning the finite rounded games,
implies that $\A$ has an ultrapower, that wins the $\omega$ rounded game.
But in this case, the atom structure $\alpha$ of this ultrapower holds an algebra
in $\Nr_n\CA_{\omega}.$  Indeed, assume that \pe\ has a \ws\ in the $\omega$ rounded $H(\alpha)$.
One can then construct a generalized atomic weak set algebra of dimension $\omega$ such that the atom
structure of its full $n$ neat reduct is isomorphic to $\alpha$.

Fix some $a\in\alpha$. Using \pe\ s \ws\ in the game of neat hypernetworks, one defines a
nested sequence $N_0\subseteq N_1\ldots$ of neat hypernetworks
where $N_0$ is \pe's response to the initial \pa-move $a$, such that
\begin{enumerate}
\item If $N_r$ is in the sequence and
and $b\leq {\sf c}_lN_r(\langle f_0, f_{n-2}\rangle\ldots , x, f_{n-2})$.
then there is $s\geq r$ and $d\in\nodes(N_s)$ such
that $N_s(f_0, f_{i-1}, d, f_{n-2})=b$.
\item If $N_r$ is in the sequence and $\theta$ is any partial
isomorphism of $N_r$ then there is $s\geq r$ and a
partial isomorphism $\theta^+$ of $N_s$ extending $\theta$ such that
$\rng(\theta^+)\supseteq\nodes(N_r)$.
\end{enumerate}
Now let $N_a$ be the limit of this sequence.
This limit is well-defined since the hypernetworks are nested.
We shall show that $N_a$ is the base of a weak set algebra having unit  $^{\omega}N_a^{(p)}$,
for some fixed sequence $p\in {}^{\omega}N$, for which there exists a homomorphism $h$ from $\A\to \wp(N_a)$
such that $h(a)\neq 0$.

Let $\theta$ be any finite partial isomorphism of $N_a$ and let $X$ be
any finite subset of $\nodes(N_a)$.  Since $\theta, X$ are finite, there is
$i<\omega$ such that $\nodes(N_i)\supseteq X\cup\dom(\theta)$. There
is a bijection $\theta^+\supseteq\theta$ onto $\nodes(N_i)$ and $j\geq
i$ such that $N_j\supseteq N_i, N_i\theta^+$.  Then $\theta^+$ is a
partial isomorphism of $N_j$ and $\rng(\theta^+)=\nodes(N_i)\supseteq
X$.  Hence, if $\theta$ is any finite partial isomorphism of $N_a$ and
$X$ is any finite subset of $\nodes(N_a)$ then
\begin{equation}\label{eq:theta}
\exists \mbox{ a partial isomorphism $\theta^+\supseteq \theta$ of $N_a$
 where $\rng(\theta^+)\supseteq X$}
\end{equation}
and by considering its inverse we can extend a partial isomorphism so
as to include an arbitrary finite subset of $\nodes(N_a)$ within its
domain.
Let $L$ be the signature with one $n$ -ary predicate symbol ($b$) for
each $b\in\alpha$, and one $k$-ary predicate symbol ($\lambda$) for
each $k$-ary hyperlabel $\lambda$. We are working in usual first order logic.

Here we have a sequence of variables of order type $\omega$, and two 'sorts' of formulas,
the $n$ predicate symbols uses only $n$ variables, and roughly
the $n$ variable formulas built up out of the first $n$ variables will determine the neat reduct, the $k$ ary predicate symbols
wil determine algebras of higher dimensions as $k$ gets larger.
This process will be interpreted in an infinite weak set algebra with base $N_a$, whose elements are
those  assignments satisfying such formulas.

For fixed $f_a\in\;^\omega\!\nodes(N_a)$, let
$U_a=\set{f\in\;^\omega\!\nodes(N_a):\set{i<\omega:g(i)\neq
f_a(i)}\mbox{ is finite}}$.
Notice that $U_a$ is weak unit (a set of sequences agreeing cofinitely with a fixed one.)


We can make $U_a$ into the universe an $L$ relativized structure $\c N_a$;
here relativized means that we are only taking those assignments agreeing cofinitely with $f_a$,
we are not taking the standard square model.
However, satisfiability  for $L$ formulas at assignments $f\in U_a$ is defined the usual Tarskian way, except
that we use the modal notation, with assignments on the left:
For $b\in\alpha,\;
l_0, \ldots l_{n-1}, i_0 \ldots, i_{k-1}<\omega$, \/ $k$-ary hyperlabels $\lambda$,
and all $L$-formulas $\phi, \psi$, let
\begin{eqnarray*}
\c N_a, f\models b(x_{l_0}\ldots  x_{l_{n-1}})&\iff&N_a(f(l_0),\ldots  f(l_{n-1}))=b\\
\c N_a, f\models\lambda(x_{i_0}, \ldots,x_{i_{k-1}})&\iff&  N_a(f(i_0), \ldots,f(i_{k-1}))=\lambda\\
\c N_a, f\models\neg\phi&\iff&\c N_a, f\not\models\phi\\
\c N_a, f\models (\phi\vee\psi)&\iff&\c N_a,  f\models\phi\mbox{ or }\c N_a, f\models\psi\\
\c N_a, f\models\exists x_i\phi&\iff& \c N_a, f[i/m]\models\phi, \mbox{ some }m\in\nodes(N_a)
\end{eqnarray*}
For any $L$-formula $\phi$, write $\phi^{\c N_a}$ for the set of all $n$ ary assignments satisfying it; that is
$\set{f\in\;^\omega\!\nodes(N_a): \c N_a, f\models\phi}$.  Let
$D_a = \set{\phi^{\c N_a}:\phi\mbox{ is an $L$-formula}}.$
Then this is the universe of the following weak set algebra
\[\c D_a=(D_a,  \cup, \sim, {\sf D}_{ij}, {\sf C}_i)_{ i, j<\omega}\]
then  $\c D_a\in\RCA_\omega$. (Weak set algebras are representable).

Now we consider the extra dimensions. Let $\phi(x_{i_0}, x_{i_1}, \ldots, x_{i_k})$ be an arbitrary
$L$-formula using only variables belonging to $\set{x_{i_0}, \ldots,
x_{i_k}}$.  Let $f, g\in U_a$ (some $a\in \alpha$) and suppose that $\{(f(i_j), g(i_j): j\leq k\}$
is a partial isomorphism of $N_a$, then one can easily prove by induction over the
quantifier depth of $\phi$ and using (\ref{eq:theta}), that
\begin{equation}
\c N_a, f\models\phi\iff \c N_a,
g\models\phi\label{eq:bf}\end{equation}

Let $\C=\prod_{a\in \alpha} \c D_a$.  Then  $\C\in\RCA_\omega$, and $\C$ is the desired generalized weak set algebra.
Note that unit of $\C$ is the disjoint union of the weak spaces.
 Then, it is not hard to show that,  $\alpha\cong \At\Nr_n\C.$
\end{proof}

\subsection{Neat and $\Ra$ reducts of cylindric algebras}

Hirsch \cite{r} proved that the class $\Ra\CA_k$ is not elementary using a rainbow algebra. Here we show that this class is not elementary 
using an entirely different construction, invented by the author with a 
precursor by Andr\'eka and N\'emeti, that is appropriate for constructing {\it complete} elementary subalgebras of 
algebras in $\Nr_n\CA_{\omega}$ that are not even in $\Nr_n\CA_{n+1}$.
The algebra constructed by Hirsch witnessing that $\Ra\CA_5$ is not elementary,
is not a complete subalgebra of the full $\Ra$ reduct and indeed he asks whether there is one that is.
Here we show that there could be one if a certain cylindric algebra term using $5$ variables 
(that is a $5$ variable first order formula).  

In our next theorem on neat reducts we use the original Monk's algebras 
because they have a neat embedding property, suitable for our purposes.

\begin{theorem} For any infinite ordinal $\alpha$, the class $\Nr_n\CA_{\alpha}$ is not elementary, but it is 
psuedo elementary; further more
${\sf Up Ur}\Nr_n\CA_{\alpha}$ is not finitely axiomatizable
\end{theorem}
\begin{proof}
For $n<m <\omega$, the characterisation is easy. One defines the class $\Nr_n\CA_m$ in a two sorted language. 
The first sort for the $n$ dimensional cylindric algebra the second for the $m$  dimensinal cylindric algebra. The signature of the defining theory 
includes an injective  function $I$ from sort one to sort two and includes a sentence requiring 
that $I$ respects the operations and a sentence saying that  an element of the second sort 
say $y$ satisfies  $\bigvee_{n\leq i<m} c_iy=y$, iff there exists $x$ of sort one 
such that $y=I(x)$ so that $I$ is a bijection.
 
Assume that $n$ is still finite, we first show that for any infinite $\alpha$, $\Nr_n\CA_{\omega}=\Nr_n\CA_{\alpha}$. Let $\A\in \Nr_n\CA_{\omega}$, 
so that $\A=\Nr_n\B'$, $\B'\in \CA_{\omega}$. Let $\B=\Sg^{\B'}A$. Then $\B\in \Lf_{\omega}$, and $\A=\Nr_n\B$. 
But $\Lf_{\omega}=\Nr_{\omega}\Lf_{\alpha}$ and we are done.
To show that $\Nr_n\CA_{\omega}\subseteq \Nr_n\RCA_{\omega}$, let $\A\in \Nr_n\CA_{\omega}$, then by the above argument
there exists  
then $\B\in \Lf_{\omega}$ such that $\A=\Nr_n\B$. by $\Lf_{\omega}\subseteq \RCA_{\omega},$ we are done. 

It is known that class $\Nr_n\CA_{\omega}$ is not elementary. In fact, there is an algebra $\A\in \Nr_n\CA_{\omega}$ 
having a complete subalgebra $\B$, and $\B\notin \Nr_n\CA_{n+1},$ this will be proved below.

Now assume that $m$ is infinite. Here if $y$ is in the $n$ dimensional 
cylindric algebra then we cannot express $c_i=y$ for all $i\in\omega\sim n$, like we did
when $m$ is finite, so we have to think differently.

To show that it is pseudo-elementary, we use a three sorted defining theory, with one sort for a cylindric algebra of dimension $n$ 
$(c)$, the second sort for the Boolean reduct of a cylindric algebra $(b)$
and the third sort for a set of dimensions $(\delta)$. We use superscripts $n,b,\delta$ for variables 
and functions to indicate that the variable, or the returned value of the function, 
is of the sort of the cylindric algebra of dimension $n$, the Boolean part of the cylindric algebra or the dimension set, respectively.
The signature includes dimension sort constants $i^{\delta}$ for each $i<\omega$ to represent the dimensions.
The defining theory for $\Nr_n\CA_{\omega}$ incudes sentences demanding that the consatnts $i^{\delta}$ for $i<\omega$ 
are distinct and that the last two sorts define
a cylindric algenra of dimension $\omega$. For example the sentence
$$\forall x^{\delta}, y^{\delta}, z^{\delta}(d^b(x^{\delta}, y^{\delta})=c^b(z^{\delta}, d^b(x^{\delta}, z^{\delta}). d^{b}(z^{\delta}, y^{\delta})))$$
represents the cylindric algebra axiom ${\sf d}_{ij}={\sf c}_k({\sf d}_{ik}.{\sf d}_{kj})$ for all $i,j,k<\omega$.
We have have a function $I^b$ from sort $c$ to sort $b$ and sentences requiring that $I^b$ be injective and to respect the $n$ dimensional 
cylindric operations as follows: for all $x^r$
$$I^b({\sf d}_{ij})=d^b(i^{\delta}, j^{\delta})$$
$$I^b({\sf c}_i x^r)= {\sf c}_i^b(I^b(x)).$$
Finally we require that $I^b$ maps onto the set of $n$ dimensional elements
$$\forall y^b((\forall z^{\delta}(z^{\delta}\neq 0^{\delta},\ldots (n-1)^{\delta}\rightarrow c^b(z^{\delta}, y^b)=y^b))\leftrightarrow \exists x^r(y^b=I^b(x^r))).$$

In this case we need a fourth sort. We leave the details to the reader.

In all cases, it is clear that any algebra of the right type is the first sort of a model of this theory. 
Conversely, a model for this theory will consist of an $n$ dimensional cylindric algebra type (sort c), 
and a cylindric algebra whose dimension is the cardinality of 
the $\delta$-sorted elements, which is at least $|m|$. Thus the first sort of this model must be a neat reduct.

For the second part Monk's original algebras do the job, by observing two things. First that these algebras are actually {\it full} neat reducts, 
and second that the class of neat reducts is closed under ultrproducts, in fact, neat reducts commute with forming ultraproducts
\end{proof}

We shall prove (the second item (modulo the existence of a $k$ witness)  answers a question of Hirsch \cite{r}.)

\begin{theorem} Let $\K$ be any of cylindric algebra, polyadic algebra, with and without equality, or Pinter's substitution algebra.
We give a unified model theoretic construction, to show the following:
\begin{enumarab}
\item For $n\geq 3$ and $m\geq 3$, $\Nr_n\K_m$ is not elementary, and $S_c\Nr_n\K_{\omega}\nsubseteq \Nr_n\K_m.$
\item Assume that there exists a $k$-witness. For any $k\geq 5$, $\Ra\CA_k$ is not elementary
and $S_c\Ra\CA_{\omega}\nsubseteq \Ra\CA_k$.
\end{enumarab}
\end{theorem}

We  now show that are strongly representable uncountable atom structures that are not in $\Nr_n\CA_{n+1}$, least in
$\Nr_n\CA_{\omega}$. 
However, we prove a more general result that applies to many cylindric-like algebras, as well as to
relation algebras except that for relation algebras, we do not know what is the least $k$ such the constructed algebra is not in $\Ra\CA_k$, 
but we conjecture that is $5$. 

A $k$ witness which is a $\CA_k$ term with special properties will be defined below. 
For $\CA$ and its relatives the idea is very much like that in \cite{MLQ}, the details implemented, in each separate case,
though are significantly distinct, because we look for terms not in the clone of operations
of the algebras considered; and as much as possible, we want these to use very little spare dimensions, hopefully just one.

The relation algebra part is more delicate. We shall construct an atomic relation algebra $\A\in \Ra\CA_{\omega}$ with a complete subalgebra $\B$,
such that $\B\notin \Ra\CA_k$, and $\B$ is elementary equivalent to $\A.$ (In fact, $\B$ will be an elementary subalgebra of $\A$.)
Futhermore, $\B$ is strongly representable. (By elementarity it is atomic)
We work with $n=3$. One reason, is that for higher dimensions the proof is the same.
Another one is that in the relation algebra case, we do not need more dimensions.

Roughly, the idea is to use an uncountable atomic cylindric algebra $\A\in \Nr_3\CA_{\omega}$,  $\A$ wil be strongly representable (its completion, 
the complex algebra of its atom structure is representable), together
with a finite atom structure of another simple cylindric algebra, that is also (strongly) representable.

The former algebra will be a set algebra based on a homogeneous model, that admits elimination of quantifiers
(hence will be a full neat reduct). Such a model is constructed using  Fraisse's methods of building models by amalgamating smaller parts.

The Boolean reduct of $\A$ can be viewed as a finite direct product of the of disjoint Boolean relativizations of $\A$,  which are also atomic.
Each component will be still uncountable; the product will be indexed by the elements of the atom structure.
The language of Boolean algebras can now be expanded by constants also indexed by the atom structure,
so that $\A$ is first order interpretable in this expanded structure $\P$ based on the finite atomic Boolean product.
The interpretation here is one dimensional and quantifier free.

The $\Ra$ reduct of $\A$ be as desired; it will be a full $\Ra$ reduct of a full neat reduct of an $\omega$ 
dimensional algebra, hence an $\Ra$ reduct
of an $\omega$ dimensional algebra, and it has a
complete elementary equivalent subalgebra not in
$\Ra\CA_k$. (This is the same idea for $\CA$, but in this case, and the other cases of its relatives, one spare dimension suffices.)

This {\it elementary subalgebra} is obtained from $\P$, by replacing one of the components of the product with an elementary
{\it countable} Boolean subalgebra, and then giving it the same interpretation.
First order logic will not see this cardinality twist, but a suitably chosen term
$\tau_k$ not term definable in the language of relation algebras will, witnessing that the twisted algebra is not in $\Ra\CA_k$.

For $\CA$'s and its relatives,
we are lucky enough to have $k$ just $n+1,$
proving the most powerful result.

\begin{definition}
Let $k\geq 4$. A $k$ witness $\tau_k$ is $m$-ary term of $\CA_k$ with rank $m\geq 2$ such 
that $\tau_k$ is not definable in the language of relation algebras (so that $k$ has to be $\geq 4$)
and for which there exists a term $\tau$ expressible in the language of relation algebras, such that
$\CA_k\models \tau_k(x_1,\ldots x_m)\leq \tau(x_1,\ldots x_m).$ (This is an implication between two first order formulas using $k$-variables).

Furthermore, whenever, $\A\in {\sf Cs}_k$ (a set algebra of dimension $k$) is uncountable,
and $R_1,\ldots R_m\in A$  are such that at least one of them is uncountable,
then $\tau_k^{\A}(R_1\ldots R_m)$ is uncountable as well.
\end{definition}

A variant of the following lemma, is available in \cite{Sayed} with a sketch of proof; it is fully 
proved in \cite{MLQ}. If we require that a (representable) algebra be a neat reduct,
then quantifier elimination of the base model guarantees this, as indeed illustrated below. 

However, in \cite{Sayed} different relations symbols only had distinct interpretations, meaning that they could have non-empty intersections;
here we strengthen
this condition to that they have {\it disjoint} interpretatons. We need this stronger
condition to show that our constructed algebras
are atomic.

\begin{lemma} Let $V=(\At, \equiv_i, {\sf d}_{ij})_{i,j<3}$ be a finite cylindric atom structure,
such that $|\At|\geq |{}^33.|$
Let $L$ be a signature consisting of the unary relation
symbols $P_0,P_1,P_2$ and
uncountably many tenary predicate symbols.
For $u\in V$, let $\chi_u$
be the formula $\bigwedge_{u\in V}  P_{u_i}(x_i)$.
Then there exists an $L$-structure $\M$ with the following properties:
\begin{enumarab}

\item $\M$ has quantifier elimination, i.e. every $L$-formula is equivalent
in $\M$ to a boolean combination of atomic formulas.

\item The sets $P_i^{\M}$ for $i<n$ partition $M$, for any permutation $\tau$ on $3,$
$\forall x_0x_1x_2[R(x_0,x_1,x_2)\longleftrightarrow R(x_{\tau(0)},x_{\tau(1)}, x_{\tau(2)}],$

\item $\M \models \forall x_0x_1(R(x_0, x_1, x_2)\longrightarrow
\bigvee_{u\in V}\chi_u)$,
for all $R\in L$,

\item $\M\models  \forall x_0x_1x_2 (\chi_u\land R(x_0,x_1,x_2)\to \neg S(x_0,x_1,x_2))$
for all distinct tenary $R,S\in L$,
and $u\in V.$

\item For $u\in V$, $i<3,$
$\M\models \forall x_0x_1x_2
(\exists x_i\chi_u\longleftrightarrow \bigvee_{v\in V, v\equiv_iu}\chi_v),$

\item For $u\in V$ and any $L$-formula $\phi(x_0,x_1,x_2)$, if
$\M\models \exists x_0x_1x_2(\chi_u\land \phi)$ then
$\M\models
\forall x_0x_1x_2(\exists x_i\chi_u\longleftrightarrow
\exists x_i(\chi_u\land \phi))$ for all $i<3$

\end{enumarab}
\end{lemma}
\begin{proof}\cite{MLQ}
\end{proof}
\begin{lemma}\label{term}
\begin{enumarab}

\item For $\A\in \CA_3$ or $\A\in \SC_3$, there exist
a unary term $\tau_4(x)$ in the language of $\SC_4$ and a unary term $\tau(x)$ in the language of $\CA_3$
such that $\CA_4\models \tau_4(x)\leq \tau(x),$
and for $\A$ as above, and $u\in \At={}^33$,
$\tau^{\A}(\chi_{u})=\chi_{\tau^{\wp(^33)}(u).}$

\item For $\A\in \PEA_3$ or $\A\in \PA_3$, there exist a binary
term $\tau_4(x,y)$ in the language of $\SC_4$ and another  binary term $\tau(x,y)$ in the language of $\SC_3$
such that $PEA_4\models \tau_4(x,y)\leq \tau(x,y),$
and for $\A$ as above, and $u,v\in \At={}^33$,
$\tau^{\A}(\chi_{u}, \chi_{v})=\chi_{\tau^{\wp(^33)}(u,v)}.$


\end{enumarab}
\end{lemma}

\begin{proof}

\begin{enumarab}

\item For all reducts of polyadic algebras, these terms are given in \cite{FM}, and \cite{MLQ}.
For cylindric algebras $\tau_4(x)={}_3 s(0,1)x$ and $\tau(x)=s_1^0c_1x.s_0^1c_0x$.
For polyadic algebras, it is a little bit more complicated because the former term above is definable.
In this case we have $\tau(x,y)=c_1(c_0x.s_1^0c_1y).c_1x.c_0y$, and $\tau_4(x,y)=c_3(s_3^1c_3x.s_3^0c_3y)$.

\item  We omit the construction of such terms. But from now on, we assme that they exist.
\end{enumarab}
\end{proof}

\begin{theorem}
\begin{enumarab}
\item There exists an atomic $\A\in \Nr_3\QEA_{\omega}$
with an elementary equivalent cylindric  uncountable algebra $\B$ which is strongly representable, and its $\SC$ reduct is not in $\Nr_3\SC_4$.
Furthermore, the latter is a complete subalgebra of the former.

\item Assume that there is $k$ witness. Then there exists an 
atomic relation algebra $\A\in \Ra\CA_{\omega}$, with an elementary equivalent uncountable relation algebra $\B$, that is strongly 
representable and is not in $\Ra\CA_k$.
Furthermore, the latter is a complete subalgebra of the former.
\end{enumarab}
\end{theorem}

\begin{proof} Let $\L$ and $\M$ as above. Let
$\A_{\omega}=\{\phi^M: \phi\in \L\}.$
Clearly $\A_{\omega}$ is a locally finite $\omega$-dimensional ylindric set algebra.
The proof for $\CA$s; and its relatives are very similar. Let us prove it for $\PEA$. Here we have to add a condition to our constructed model.
We assume that the relation symbols are indexed by an uncountable set $I$.
and that  there is a group structure on $I$, such that for distinct $i\neq j\in I$,
we have $R_i\circ R_j=R_{i+j}$.
We take $\At=({}^33, \equiv_i, \equiv_{ij}, d_{ij})$, where
for $u,v\in \At$, $u\equiv_i v$ iff $u$ and $v$ agree off $i$ and $v\equiv_{ij}u$ iff $u\circ [i,j]=v$. 
We denote $^33$ by $V$.

By the symmetry condition we have $\A$ is a $\PEA_3$, and
$\A\cong \Nr_3\A_{\omega}$, the isomorphism is given by
$\phi^{\M}\mapsto \phi^{\M}.$ In fact, $\A$ is not just a polyadic equality algebras, it is also closed under al first order definable 
operations using extra dimensions for quantifier elimination in $\M$ guarantees that this map is onto, so that $\A$ is the full  neat reduct.
For $u\in {}V$, let $\A_u$ denote the relativisation of $\A$ to $\chi_u^{\M}$
i.e $\A_u=\{x\in A: x\leq \chi_u^{\M}\}.$ Then $\A_u$ is a Boolean algebra.
Furthermore, $\A_u$ is uncountable and atomic for every $u\in V$
because by property (iv) of the above lemma,
the sets $(\chi_u\land R(x_0,x_1,x_2)^{\M})$, for $R\in L$
are disjoint of $\A_u$. It is easy to see that $A_u$ is actually isomorphic to the finie co-finite Boolean algebra on  a set of cardinality $I$.

Define a map $f: \Bl\A\to \prod_{u\in {}V}\A_u$, by
$f(a)=\langle a\cdot \chi_u\rangle_{u\in{}V}.$
We expand the language of the Boolean algebra $\prod_{u\in V}\A_u$ by constants in
and unary operations, in such a way that
$\A$ becomes interpretable in the expanded structure. 

Let $\P$ denote the 
following structure for the signature of boolean algebras expanded
by constant symbols $1_u$, $u\in {}V$ and ${\sf d}_{ij}$, and unary relation symbols
${\sf s}_{[i,j]}$ for $i,j\in 3$: 

\begin{enumarab}
\item The Boolean part of $\P$ is the boolean algebra $\prod_{u\in {}V}\A_u$,

\item $1_u^{\P}=f(\chi_u^{\M})=\langle 0,\cdots0,1,0,\cdots\rangle$ 
(with the $1$ in the $u^{th}$ place)
for each $u\in {}V$,

\item ${\sf d}_{ij}^{\P}=f({\sf d}_{ij}^{\A})$ for $i,j<\alpha$.

\item ${\sf s}_{[i,j]}^{\P}(x)= {\sf s}_{[i,j]}^{\P}\langle x_u: u\in V\rangle= \langle x_{u\circ [i,j]} : u\in V\rangle.$

\end{enumarab}

Define a map $f: \Bl\A\to \prod_{u\in {}V}\A_u$, by
$$f(a)=\langle a\cdot \chi_u\rangle_{u\in{}V}.$$

Then there are quantifier free formulas
$\eta_i(x,y)$ and $\eta_{ij}(x,y)$ such that
$\P\models \eta_i(f(a), b)$ iff $b=f(c_i^{\A}a)$ and
$\P\models \eta_{ij}(f(a), b)$ iff $b=f(s_{[i,j]}a).$
The one corresponding to cylindrifiers is exactly like the $\CA$ case, the one corresponding to substitutions in $y={\sf s}_{[i,j]}x.$
Now, like the $\CA$ case, $\A$ is interpretable in $\P$, and indeed the interpretation is one dimensional and quantifier free.

For $v\in V$, let $\B_v$ be a complete countable elementary subalgebra of $\A_v$.
Then proceed like the $\CA$ case, except that we take a different product, since we have a different atom structure, with unary relations
for substitutions:
Let $u_1, u_2\in V$ and let $v=\tau(u_1,u_2)$, as given in the above lemma.
Let $J=\{u_1,u_2, s_{[i,j]}v: i\neq  j<3\}$.
Let  $\B=((\A_{u_1}\times \A_{u_2}\times \B_{v}\times \prod_{i,j<3, i\neq j} \B_{s_{[i,j]}v}\times \prod_{u\in V\sim J} \A_u), 1_u, d_{i,j}, {\sf s}_{i,j}x)$
inheriting the same interpretation. Then by the Feferman Vaught theorem, 
which says that replacing a component in a possibly infinite product by  elementary equivalent
algebra, then the resulting new product is elementary equivalent to the original one, so that $\B\equiv \P$, 
hence $\B\equiv \A$. (If a structure is interpretable in another structute then any structure
elementary equivalent to the former structure is elementary equivalent to the last). 
Notice to  that $\B$ is atomic, because $\P$ is, and atomicity is a first order property.

Now we show that $\B$ is strongly representable. The easiest way to do it, is to show that \pe\ has a winning strategy in all finite rounded atomic games.
But $\B$ is easily seen to be completely representable, hence \pe\ can indeed win the $\omega$ round game. 
Hence it can win the finite ones, and this  makes it strongly representable.
Then $\At\B$ is strongly representable.

In our new product we made all the permuted versions of $\B_v$ countable, so that $B_v$ {\it remains} countable,
because substitutions corresponding to transpositions
are present in our signature, so if one of the permuted components is uncountable, then $\B_{v}$ would be uncountable, and we do not want that.

The contradiction follows from the fact that had  $\B$ been a neat reduct, say $\B=\Nr_3\D$
then the term $\tau_3$ as in the above lemma, using $4$ variables, evaluated in $\D$ will force the component $\B_v$ to be uncountable,
which is not the case by construction, indeed $\tau_3^{\D}(f(R_i), f(R_j))=f(R_{i+j})$.

For the second part; for relation algebras.
The $\Ra$ reduct of $\A$ is a generalized reduct of $\A$, hence $\P$ is first order interpretable in $\Ra\A$, as well.
It follows that there are closed terms and  a unary relation symbol, and formula $\eta$, and $\mu$ built out of these closed terms and unary
relation symbol such that
$\P\models \eta(f(a), b, c)\text { iff }b= f(a\circ^{\Ra\A} c),$ and $\P\models \mu(f(a),b)\text { iff }b=\breve{a}$
where the composition is taken in $\Ra\A$. The former formula is built, like cylindrifiers from only closed terms, $1_u$, $u\in \At$
while converse is defined by 
the unary relation symbol.
Here $\At$ defined depends on $\tau_k$ and $\tau$, so we will not specify it any further,
we just assume that it is finite.

As before, for each $u\in \At$, choose any countable Boolean elementary
complete subalgebra of $\A_{u}$, $\B_{u}$ say.
Le $u_i: i<m$ be elements in $\At$, and let $v=\tau(u_1,\ldots u_m)$.
Let $$\B=((\prod_{u_i: i<m}\A_{u_i}\times \B_{v}\times \times \B_{\breve{v}}\times \prod_{u\in {}V\smallsetminus \{u_1,\ldots u_m, v, \breve{v}\}}
\A_u), 1_u, R, Id) \equiv$$
$$(\prod_{u\in V} \A_u, 1_u, R, Id)=\P.$$

Let $\B$ be the result of applying the interpretation given above to $Q$.
Then $\B\equiv \Ra\A$ as relation  algebras, furthermore $\Bl\B$ is a complete subalgebra of $\Bl\A$.
Now we use essentially the same argument. We force the $\tau(u_1,\ldots u_m)$
component together with its permuted versions (because we have converse) countable;
the resulting algebra will be a complete elementary subalgebra of the original one, but $\tau_k$
will force our twisted countable component to be uncountable, arriving at a contradiction.

In more detail, assume for contradiction that $\B=\Ra\D$ with $\D\in \CA_k$.
Then $\tau_k^{\D}(f(\chi_{u_1}),\ldots f(\chi_{u_n}))$, is uncountable in $\D$.
Because $\B$ is a full $\RA$ reduct,
this set is contained in $\B.$  For simplicity assume that $\tau^{\Cm\At}(u_1\ldots u_m)=Id.$
On the other hand, for $x_i\in B$, with $x_i\leq \chi_{u_i}$, let $\bar{x_i}=(0\ldots x_i,\ldots)$ 
with $x_i$ in the $uth$ place.
Then we have
$$\tau_k^{\D}(\bar{x_1},\ldots \bar{x_m})\leq \tau(\bar{x_1}\ldots \bar{x_m})\in \tau(f(\chi_{u_1}),\ldots f({\chi_{u_m}}))
=f(\chi_{\tau(u_1\ldots u_m)})=f(\chi_{Id}).$$
But this is a contradiction, since  $\B_{Id}=\{x\in B: x\leq \chi_{Id}\}$ is  countable and $f$ is 
a Boolean isomorphism.
\end{proof}

We  note  that our constructed relation algebra $\B$, $\B$ is  a complete subalgebra of $\A$. This reproves  a  result of Robin Hirsch
and answers a question of  Robin Hirsch.
The  result of Hirsch is that the class $\Ra\CA_k$ is
not elementary, and the answer to his question is that $\Ra\CA_k\subset S_c\Ra\CA_k$ for $k\geq  5$.

For a class $\K$ with a Boolean reduct,  write $\K\cap \At$ for the class of atomic algebras are in $\K$, 
The former is elementary iff the latter is.
We define a new class of atomic cylindric algebras. An atomic algebra is first order definable if the first order 
definable algebra on its atom structure
is representable, we denote this class by ${\sf FOCA_n}$.

Summarizing the above proof:
Let $\rho^{\infty}_m$ be the sentence in $L_{\omega_1,\omega},$ that reflects 
that \pe\ has a \ws\ on the $\omega$ rounded game involving $m$ pebbles. Let 
$\rho_k$ be the higher order formula that reflects that \pe\ has a \ws\ in $H_k,$ 
the game $H$ truncated to $k$. (We do not know how high is higher).
Let ${\sf Op}K=\{\A\in K: \text {such that $\A$ is countable and atomic}\}$,
then 
$${\sf Op}{\sf Mod}\{\rho_k: k\in \omega\}\subseteq {\sf Op}{\sf UpUr}\Nr_n\CA_{\omega},$$ 
and 
$$S_c\Nr_n\CA_{n+m}\subseteq {\sf Op}{\sf Mod}\{\rho^{\infty}_m\}$$
So our proof in item one, can be summarized as that there is a countable rainbow atomic algebra 
that satisfies $\rho_k$ for every $k\in \omega$, hence the Lyndon conditions,
but does not satisfy $\rho_m^{\infty}.$

Now we introduce a new elementary class of representable algebras, and that is 
${\sf Mod}\{\rho_k: k\in \omega\}$, call it ${\sf SLCA_n}$ short for {\it strong Lyndon algebras}.
So we have ${\sf SLCA_n}\subseteq {\sf  LCA_n}.$

Let $\K$ be the class of atomic representable algebras having $NS$, and ${\sf L}$ be the class of atomic 
representable algebras having  $NS$ the unique neat embedding propery \cite{Sayedneat}.
Obviously, the latter is contained in the former, and both are contained in $\Nr_n\CA_{\omega}.$
The next theorem shows that there are a plathora of very interesting classes between the 
atomic algebras in the amalgamation base of ${\sf RCA_n}$ and atomic algebras in 
${\sf RCA_n}.$ Some are elementary, some are not.

\begin{theorem}
We have the following inclusions (note that $\At$ commutes with ${\sf UpUr})$:
$${\sf L}\subseteq \K\subseteq \Nr_n\CA_{\omega}\cap \At\subset {\sf UpUr}\Nr_n\CA_{\omega}\cap \At$$
$$\subseteq {\sf Up Ur}S_c\Nr_n\CA_{\omega}\cap \At={\sf UpUr}{\sf CRA_n}$$
$$\subseteq {\sf SLCA_n} \subseteq {\sf LCA}_n\subset {\sf SRCA_n}\subset{\sf FOCA_n}\subset {\sf UpUr}{\sf SRCA}_n$$
$$\subseteq S\Nr_n\CA_{\omega}\cap \At={\sf WRCA_n}= \RCA_n\cap \At.$$
\end{theorem}
\begin{proof}
Items (5) and (7) uses ideas of Ian Hodkinson and Robin Hirsch.
\begin{enumarab}
\item  By the inclusion in \cite{Sayedneat}, and the equivalence in \cite{IGPL}, 
the first inclusion is witnessed by an atomic algebra that lies in the amalgamation base of $\RCA_n$,
but not in the super amalgamation base of  $\RCA_n$. The second is witnessed by an atomic 
algebra in $\Nr_n\CA_{\omega}$, that is not in the strong amalgamation base of $\RCA_n$.

\item The third inclusion is witnessed by the algebra $\B$ constructed above which is uncountable. $\B$ is also completely representable,
so it witnesses the strictness of fourth inclusion {\it without} the elementary closure operator (with it we do not know whether the inclusion is strict). 

\item The fifth  inclusion is witnessed by the rainbow algebra construced on an atom structure 
for which \pe\ can win $H_k$, for all finite $k$, but cannot win $F^m$. 

\item The sixth we do not know whether the game coded by $\sigma_k$ is strictly harder than that coded by
$\rho_k$, as far as \pe\ is concerned. 

\item We now provide a concrete example of algebra that there is a strongly representable algebra that fails infinitely many Lyndon conditions.
(We know that one exists becuase ${\sf LCA_n}$ is elementary and it is contained in ${\sf RSA_n}$ which is not elementary.
Anti-Monk algebras have affinity to Monk's algebras. In fact, they are both based on   graphs that forbid {\it independent} 
monochromatic triangles (not all). 
Now let $\Gamma$ be any graph with infinite chromatic number, 
and large enough finite girth.  Let $m$ be also large enough so that any $3$ colouring of the edges of a complete graph 
of size $m$ must contain a monochromatic triangle; this $m$ exists by Ramseys's theorem.

Then $\M(\Gamma)$, the complex algebra constructed on $\Gamma$, as defined in \cite{HHbook} 
will be representable, hence $\rho(I(\Gamma)$ wil be strongly representable, but it will fail $\rho_k$ for all $k\geq m$. The idea is 
that \pa\ can win in the $m$ rounded atomic game coded by $\rho_m$, by forcing a forbidden monochromatic triangle.
We can assume that $m>n$ wheren $n$ is the dimension.
Let $N$ be an atomic network with $m$ nodes. Choose a set $X$ of  $max\{n,6\}$ nodes of $\Gamma$, 
such that the colour of $N(\bar{x})$ is constant say 
$r$, for every hyperedge of $X$.  
For $\bar{x}\in X$, of distinct elements, let $v(\bar{x})\in \Gamma$ 
be such that $N(\bar{x})=r$, and let $\Delta$ be the induced subgraph with nodes $\{v{\bar{x}}: \bar{x}\in X\}$ of $\Gamma$.
Since the girth is sufficiently large, $\Delta$ is $2$ colourable and its nodes can be partitioned into two distinct sets, each independent 
and monochromatic. But any $2$ colouring of the edges of a complete graph of  size $\geq 6$, has  {\it an independent monochromatic triangle.}

\item Let $\A$ be the $\bold G$ Monk algebra constructed in our first blow up and blur construction, or the algebra based 
on the rainbow construction proving that
$S\Nr_{n}\CA_{n+k}$ is not atom canonical. 
Recall that such algebras were defined using first order formulas, the first in a Monk's signature, the second in the rainbow signature 
(the latter is first order since we had only 
finitely many greens). A distinction worthwhile highlighting here, is that the first algebra based on an infinite graph 
with finite chromatic number, and that is why the complex algebra is not 
representable. The second rainbow algebra is based on a complete infinite irreflexive graphs, 
the graph of reds. Then $\A\in {\sf FOCA_n}$ but not in ${\sf SRCA_n}.$

\item We now show that ${\sf FOCA_n}\subset  {\sf WCA_n}$. Take an $\omega$ copy of the an $3$ element graph with nodes $\{1,2,3\}$ and edges
$1\to 2\to\ 3$. Then of course $\chi(\Gamma)<\infty$. Now an $\Gamma$  has a three first order definable colouring.
Since $\M(\Gamma)$ is not representable, then the algebra of first order definable sets is not representable because $\Gamma$ is interpretable in 
$\rho(\Gamma)$, the atom structure constructed from $\Gamma$ as defined in \cite{HHbook}. 
However, the term algebra can be easily seen to be representable, since it consists only of finite and cofinite subsets of the atom 
structure \cite{HHbook}.
\end{enumarab}
\end{proof}

\subsection{Open questions}

For a class $K$ having a Boolean reduct, let $K\cap \At$ denote the class of atomic algebras in $K$, notice that the first is elementary iff 
the second is.
${\sf Up}$ denotes the operation of forming ultraproducts and ${\sf Ur}$ that of forming ultraroots?.

\begin{enumarab}

\item We know that countable  completely representable algebras coincide with the class $S_c\Nr_n\CA_{\omega}\cap \At$. 
We also know that  ${\sf WRCA}_n=\RCA_n\cap \At =S\Nr_n\CA_{\omega}\cap \At$.
Can a similar characterization using neat embedings be obtained for ${\sf LCA_n}$ and ${\sf SRCA_n}$?

In particular, is it true that ${\sf UpUr} {\sf SRCA_n}={\sf WRCA}_n=S\Nr_n\CA_{\omega}\cap \At?$,
and that ${\sf Ur  CRA_n}={\sf LCA_n}$?

We know that there are bad algebras that converge to good ones. The question is, is every bad algebra a limit of good ones.
(This reminds us of Simons amazing result, of 'representing' non representable algebras, obtaining every $\CA_3$ 
by twisting relativizing and dilating a representable one)
In other words, given a bad algebra, can we make it good via an ultraproduct or 
an ultrapower or an iteration of both?

We note that every good usual Monk algebra, namely, every $\A\in \RCA_n$ can be approached by bad ones. 
Indeed, for $\A\in S\Nr_n\CA_{\omega}$, so for $m\geq n$, let 
$\A_m=\Nr_n\B_m$, where $\A\subseteq \Nr_n\B_m$, $\B_m\in \CA_m$. Then $\A_m$ converges to $\A$.
(The $\A_m$'s are bad because they many not be representable).

\item If $\A$ is atomic and countable and $\A\in \Nr_n\CA_{\omega}$, then $\A$ is completely representable, hence strongly 
representable. In other words, the result in theorem 1.1 in \cite{ANT} is the best possible. $k$ cannot be infinite.
It is not hard to construct algebras that are strongly representable but not in $\Nr_n\CA_{\omega}$.
Any countable atomic strongly representablke algebra thgat is not completely representable will do. 
The complex algebra $\M(\Gamma)$ constructed in item (5) above where $\Gamma$ has infinite
chromatic number, with large enough girth $m$, will fail infinitely many lyndon conditions $\rho_k$ for $k\geq \omega$, 
hence will not be completely
representable. Any countable elementary subalgebra  will be as required. It will also fail 
infinitely many Lyndon conditions, hence will not be completely representable, hence wil not be in $S_c\Nr_n\CA_{\omega}$, hence wil not
be a full neat reduct of a $\CA_{\omega}$.

One way to approach this problem, is to choose a graph $\Gamma$, such that the Monk structures or 
Rainbow structures based on  $\Gamma$ has an $n$ homogeneous countable model that has quantifier 
elimination. 
This model will encode all $n$ coloured graphs (structures), namely the atoms, 
and the set algebra based on this graph (obtained by dropping assignments labelled by  one or more flexible ultrafilter or refexive nodes),
will be  representable. The term algebra wil be representable, precisely because it is {\it not} complete, so precarious joins are not there,
only finite or cofinite ones are. Even more, it will be a full neat reduct, because of quantifier elimination.
But its completion, the complex algebra of its atom structure, will not be 
representable, because for the precise reason it {\it is complete}, and precarious joins will deliver an inconsistency,  prohibiting 
a representation. 

Blow up and blur constructions seem to be the apt technique for this.

\item If $\A$ is an atomic $\CA_n$, $n\geq 3$,  and $\A\models \rho_k$, for every finite $k\geq n$, 
that is \pe\ has a \ws\ in $H_k(\At\A)$, does $\A\in \Nr_n\CA_k$?
If yes, then if $\A\models \rho_k$ for evey $k\geq n$, then $\A\in \Nr_n\CA_{k}$; in other words 
$$\cap_{n\leq k<\omega}\Nr_n\CA_k\neq 
\Nr_n\CA_{\omega}.$$ 
(It is known that if $\A\in \cap _{n\leq k<\omega}\Nr_n\CA_k$, then $\A\in {\sf UpUr}\Nr_n\CA_{\omega}$, an unpublished result of 
Andr\'eka and N\'emeti.)
\end{enumarab}

\end{document}